%% file: main.tex
\documentclass[10pt]{amsart}
\usepackage{amsmath,amssymb, amsthm} 

\theoremstyle{plain}
    \newtheorem{thm}{Theorem}

    \newtheorem{prop}{Proposition}[section]
    
    \newtheorem{lemma}[prop]{Lemma}
    
    \newtheorem{problem}[prop]{Problem}

\theoremstyle{definition}
    
\theoremstyle{remark}
    \newtheorem{rem}[prop]{Remark}

\def\inter{\mathrm{int}}

\def\tr{{\mathrm{tr}}}

\def\uu{\mathfrak{u}}
\def\mmu{\sigma}

\def\be{\begin{equation}}
\def\ee{\end{equation}}

\def\bm{\begin{matrix}}
\def\em{\end{matrix}}

\newcommand{\C}{\mathbb{C}}\newcommand{\R}{\mathbb{R}}\newcommand{\Q}{\mathbb{Q}}\newcommand{\Z}{\mathbb{Z}}\newcommand{\N}{\mathbb{N}}
\newcommand{\T}{\mathbb{T}}

\newcommand{\cV}{\mathcal{V}}

\renewcommand{\setminus}{\smallsetminus}

\renewcommand{\Im}{\mathrm{Im}\;}
\newcommand{\SL}{\mathrm{SL}}

\newcommand{\id}{\mathit{id}}

\DeclareMathOperator*{\supp}{supp}

\def\SO{\mathrm{{SO}}}

\def\H{\mathbb{{H}}}

\def\B{\mathbf{{B}}}
\def\u{\mathbf{{u}}}

\def\Leb{\mathrm{{Leb}}}

\newcommand{\comm}[1]{}

\begin{document}

\title{On the Kotani-Last and Schr\"odinger conjectures}

\author[A.~Avila]{Artur Avila}
\address{CNRS UMR 7586,
Institut de Math\'ematique de Jussieu\\
175 rue du Chevaleret, 75013-Paris, FRANCE.
}
\address{IMPA,
Estrada Dona Castorina 110, Rio de Janeiro, Brasil} 
\urladdr{www.impa.br/$\sim$avila/}
\email{artur@math.jussieu.fr}

\begin{abstract}
In the theory of ergodic one-dimensional Schr\"odinger operators, ac
spectrum has been traditionally expected to be very rigid.  Two
key conjectures in this direction state, on one hand, that ac spectrum
demands almost periodicity of the potential, and, on the other hand, that
the eigenfunctions are almost surely bounded in the essential suport of the
ac spectrum.  We show how the repeated slow
deformation of periodic potentials
can be used to break rigidity, and disprove both conjectures.
\end{abstract}

\date{\today}

\maketitle




\input{introduction}
\input{continuum}

\input{discrete}

\end{document}

%% file: introduction.tex
\section{Introduction}

In this paper we consider one-dimensional Schr\"odinger operators,
both on the real line $\R$ and on the lattice $\Z$.  In the first case, they
act on $L^2(\R)$ and have the form
\be
(H \uu)(t)=-\frac {d^2} {dt^2} \uu(t)+V(t) \uu(t),
\ee
while in the second case they act on $\ell^2(\Z)$ and have the form
\be \label {discrete}
(H\uu)_n=\uu_{n+1}+\uu_{n-1}+V(n) \uu_n.
\ee

We are interested in the so-called ergodic case, where one considers a
measured family of potentials defined
by the evaluation of a sampling function along the orbits of a dynamical system. 
Thus, in the first (continuum) case, we have $V(t)=v(F_t(x))$, where
$F_t$ is an ergodic flow and in the second (discrete) case $V(n)=v(f^n(x))$,
where $f$ is an ergodic invertible map.  We denote the implied fixed
probability measure by $\mmu$.
We will also assume below that flows, maps, and sampling functions
are continuous in some compact phase space $X$ and that $\supp \mmu=X$.

By general reasoning, the spectrum of ergodic operators is almost surely
constant.  In general, the spectral measure is not almost surely
independent of $x \in X$, but the ac part of the spectral measure
is.  There is much work dedicated to the understanding of the ac part
of the spectral measure, with most results so far pointing to very rigid
behavior \cite {K}, \cite {DeS}, \cite {CJ} (see also \cite {R} for
recent developments regarding non-ergodic potentials).
Two natural problems in this direction are:

\begin{problem}

Does the existence of an absolutely continuous component of the spectrum
(for almost every $x \in X$) imply that the potential is almost periodic?

\end{problem}

We recall that an almost periodic potential is one that can be obtained by
evaluating a continuous sampling function along an ergodic
translation of a compact
Abelian group.  Another way of formulating this is that the dynamics has a
topological almost periodic factor, through which the sampling function
factorizes.

\begin{problem}

Are all eigenfunctions bounded, for almost every energy, with respect to
the ac part of the spectral measure (for almost every $x \in X$)?

\end{problem}

Here, by an eigenfunction associated to energy $E$ we mean a
generalized solution of $H\uu=E\uu$ (i.e., without the requirement of belonging
to $L^2(\R)$ or $\ell^2(\Z)$).

For the first problem, an affirmative answer has been explicitly conjectured
in the discrete case, in what is now known as the {\it Kotani-Last Conjecture}
(recently popularized in \cite {D}, \cite {J}, and \cite {S}, see also the
earlier \cite {KK})

For the second problem, and also in the discrete case, the affirmative
answer would be a particular case of the so-called
{\it Schr\"odinger Conjecture} (see \cite {V}, \S 1.7),
according to which eigenfunctions should be bounded for almost every energy
with respect to the ac part of the spectral measure, for any
(possibly non-ergodic) potential.  We note that in the continuum case, the
corresponding general statement was known to be false if one does not assume
the potential to be bounded: The famous
counterexample in \cite {MMG} is indeed unbounded both from above and
from below (it is also sparse, hence non-ergodic).  As it turns out, in such
setup the absolutely continuous spectrum is not constrained by any
strict ``Parseval-like'' bound on the average size
of eigenfunctions.\footnote {Particularly, moments of growth do not
have to be spread out according to the energy, and in fact in \cite {MMG}
many eigenfunctions become simultaneously large (in short bursts).}
Such a bound (\ref {spectral}) is a key
difficulty in our setup, and it is what makes it more similar to another
situation of interest (square-integrable potentials), as we will discuss.

\begin{rem}

Here is one example of how those conjectures could be used to
deduce further regularity properties.
It is known that, almost surely in the essential support of ac spectrum,
there is a pair of linearly independent complex-conjugate eigenfunctions
$\uu$, $\overline \uu$,
satisfying $|\uu(t)|=U(F_t(x))$ or $|\uu_n|=U(f^n(x))$, according to the
setting, with $U:X \to (0,\infty)$ some $L^2$-function
(depending on $E$ but independent of $x$), see \cite {DeS}.
If the dynamics is minimal
then boundedness of the
eigenfunctions implies that $U$ is in fact {\it continuous} \cite {Y},
so if the
dynamics is almost periodic then the absolute value of these eigenfunctions
is itself almost periodic.

\end{rem}

\begin{rem}

There are of course many examples of almost periodic potentials with ac
spectrum, dating from the KAM based work of Dinaburg-Sinai \cite {DS}.
KAM approaches do tend to produce bounded eigenfunctions.  In \cite {AFK},
it has been proved that
if $f$ is an irrational rotation of the circle, and $v$ is analytic, then
(up to taking some sufficiently deep renormalization) a
non-standard KAM scheme converges almost everywhere in the essential support
of the ac part of the spectral measure, so that eigenfunctions are indeed
bounded as predicted by the Schr\"odinger Conjecture.

\end{rem}

In this paper we give negative answers to both problems,
in both the discrete and the continuous setting.

\begin{thm} \label {discreteunbounded}

There exists a uniquely ergodic map, a sampling function,
and a positive measure set $\Lambda \subset \R$
such that for almost every $x$, $\Lambda$ is
contained in the essential support of the absolutely continuous spectrum,
and for every $E \in \Lambda$ and almost every $x$,
any non-trivial eigenfunction is unbounded.\footnote {Notice that, by
general reasoning, for any $E$ in the spectrum there exists always some
$x$ with a one-dimensional subspace of bounded eigenfunctions.  On the other
hand, if the dynamics is minimal, then the existence of an
unbounded eigenfunction for some $x$ implies that there are unbounded
eigenfunctions for every $x$ (with the same $E$).}

\end{thm}

\begin{thm} \label {discretewm}

There exists a weak mixing uniquely ergodic map and a
non-constant sampling function such that
the spectrum has an absolutely continuous component for every $x$.

\end{thm}

\begin{thm} \label {continuumunbounded}

There exists a uniquely ergodic flow and a sampling function,
such that
the spectrum is purely absolutely continuous
for almost every $x$, and for almost every energy in the spectrum,
and almost every $x$, any non-trivial eigenfunction is unbounded.

\end{thm}

\begin{thm} \label {continuumwm}

There exists a weak mixing uniquely ergodic flow and a
non-constant sampling function, such that
the spectrum is purely absolutely continuous for every $x$.

\end{thm}

We recall that weak mixing means the absence of a measurable almost periodic
factor.  In particular, potentials associated to non-constant
sampling functions are never almost periodic.

\begin{rem}

One may wonder whether there is some
natural condition (stronger than lack of almost periodicity) on the dynamics
that would prevent the existence of ac spectrum.  After we announced in 2009
the earliest result of this work (a less precise version of Theorem \ref
{discretewm}), Svetlana Jitomirskaya asked us whether weak mixing
would be such a condition.  Though our original
(unpublished) construction did not yield
weak mixing, the underlying mechanism could indeed be used to
answer her question, as shown in the argument we present here.

\end{rem}

\begin{rem}

Unbounded eigenfunctions can appear with or without almost periodicity:
the example provided in the proof of
Theorem \ref {continuumunbounded}
can be shown to be weak mixing (though it is not done
here), while the example provided in the proof of Theorem \ref
{discreteunbounded} is almost periodic.  In the other direction, the proofs
of Theorems \ref {discretewm} and \ref {continuumwm} (see Remarks \ref
{blaremark2}
and \ref {blaremark4}) show that bounded eigenfunctions are also
compatible with weak mixing.

\end{rem}

Our methods do give considerable more control on the continuum case (in that
we get control on the entire spectrum).  The
arguments are also much simpler.  For this reason, we first develop all
arguments in full detail for the continuum case.  We then describe more
leisurely the additional complications involved in the discrete case.

\subsection{Further perspective}

Besides its natural interest in the theory of orthogonal polynomials and
one-dimensional Schr\"odinger operators, much of the motivation behind
the Schr\"odinger Conjecture lies in its
interpretation as a generalization of the sought after
non-linear version of Carleson's Theorem
about pointwise convergence of the Fourier series of an
$L^2(\R/\Z)$ function.  Recall that this theorem
(which solved Lusin's Conjecture)
is equivalent to the statement that for any sequence of complex
numbers $\{\lambda_n\}_{n \in \N}$ with $\sum |\lambda_n|^2<\infty$, and for
almost every $\theta \in \R$, the series $\sum \lambda_n e^{2 \pi i n
\theta}$ is bounded.

One simple formulation of a (conjectural) non-linear
version of Carleson's Theorem
goes as follows: for any sequence of $\SL(2,\R)$ matrices $\{A_j\}_{j \in
\N}$ such that
\be
\sum (\ln \|A_j\|)^2<\infty
\ee
and for almost every $\theta \in \R$, the sequence
$A^{(n)}(\theta)=A_n R_\theta \cdots A_1 R_\theta$ is bounded
(here $R_\theta$ is the rotation of angle $2 \pi \theta$).  To see the
connection, notice that it is enough (by polar decomposition) to consider the
case where $A_j=D_{e^{\lambda_j}} R_{\beta_j}$, where $\beta_j \in \R$ is arbitrary
and $\lambda_j \geq 0$ satisfy $\sum |\lambda_j|^2<\infty$.
Expand
\be
\begin{pmatrix} e^{\lambda_j}&0\\0&e^{-\lambda_j}\end{pmatrix}=
\sum_{k \geq 0} \frac {\lambda_j^k} {k!}
\begin{pmatrix} 1&0\\0&(-1)^k \end{pmatrix},
\ee
and then expand the product to get
\be
A^{(n)}(\theta)=\sum_{m \geq 0} B_{m,n}(\theta),
\ee
where the coefficients of $B_{m,n}$ are homogeneous polynomials of degree
$m$ on the $\lambda_j$.  Then a direct computation gives, with
$\alpha_j=\sum_{j' \leq j} \beta_{j'}$, $B_{0,n}=R_{n \theta+\alpha_n}$ (which thus
has unit norm), while 
\be
B_{1,n}=R_{n \theta+\alpha_n} \sum_{j=1}^n \lambda_j
\begin{pmatrix} \cos 4 \pi (\alpha_j+j
\theta)&-\sin 4 \pi (\alpha_j+j\theta)\\
-\sin 4 \pi (\alpha_j+j\theta)&-\cos 4 \pi (\alpha_j+j \theta) \end {pmatrix},
\ee
so that Carleson's Theorem is equivalent to the boundedness of the
$B_{1,n}$.

One reason to hope for the almost sure boundedness of the sequence
$A^{(n)}(\theta)$ is the validity of an analogue of Parseval's
Theorem:
taking $N(A)=\ln \frac {\|A\|+\|A\|^{-1}} {2}$ (which is asymptotic to $(\ln
\|A\|)^2$ when $\|A\|$ is close to $1$), we get
\be \label {pi}
\int N(A^{(n)}(\theta)) d\theta=\sum_{j=1}^n N(A_j).
\ee
This presents a quite strict constraint to the construction of any
counterexample.  Indeed, (\ref {pi}) implies that
$\|A^{(n)}(\theta)\|$ is often bounded: any growth one sees in a certain moment
must be compensated later.  This oscillation is rather hard to achieve in
the nonlinear setting: the product of two large
$\SL(2,\R)$ matrices is usually even larger than each factor, unless
there is some rather precise alignement between their polar decompositions. 
However, any such alignement would appear likely to be
destroyed when $\theta$ changes.  (Another way to see the difficulty it to
recall that the Brownian motion in the hyperbolic plane
$\SL(2,\R)/\SO(2,\R)$ diverges linearly, while in the real line
it is recurrent.)

There is nothing sacred about the above setup (which we chose to start with
only for the transparency of the various formulas), and there are many
alternative settings where a non-linear analogue of Carleson's Theorem is
expected to hold.  The basic theme to keep in ming is the goal of
showing almost sure boundedness of
square-summable perturbations of a parametrized infinite product of
elliptic matrices, in some setting where some analogue of
Parseval's Theorem holds.  For a more detailed discussion (with slightly
different setup), see the work of Muscalu-Tao-Thiele \cite {MTT}.

\subsubsection{Schr\"odinger setting}

The eigenfunctions of discrete Schr\"odinger operators (\ref {discrete})
with eigenvalue $E$ satisfy a
second-order difference equation which can be expressed in matrix form
\be
A(E,n,n+1)
\cdot \begin{pmatrix} \uu_n \\ \uu_{n-1} \end{pmatrix}=\begin{pmatrix}
\uu_{n+1} \\ \uu_n \end{pmatrix}
\ee
where
\be
A(E,n,n+1)=\begin{pmatrix} E-V(n) & -1 \\ 1 & 0 \end{pmatrix}.
\ee
Thus writing $A(E,m,n)=A(E,n-1,n) \cdots A(E,m,m+1)$, $n>m$,
we see that the the boundedness
of eigenfunctions is equivalent to the boundedness of the sequences
$A(E,0,n)$ and $A(E,-n,0)$.

It turns out that if $V \in \ell^2(\Z)$ then the essential support of
the ac spectrum is $(-2,2)$ (a result of Deift-Killip \cite {DK}).
This is also the set of $E$ such that the
unperturbed matrices $\begin{pmatrix} E&-1\\1&0 \end{pmatrix}$ are elliptic. 
Thus the expected ``Carleson's Theorem for Schr\"odinger operators'' is just
the Schr\"odinger Conjecture restricted to potentials in $\ell^2(\Z)$. 
A partial result in this direction,
under a stronger decay condition, is obtained in
\cite {CKR} (see also \cite {MTT1} for a discussion of the limitations of
this approach in the consideration of general potentials in $\ell^2(\Z)$).

Why would one want to believe in the Schr\"odinger Conjecture for
ergodic potentials?  In our view, it is due to
the validity of an inequality, reminiscent of that implied by
Parseval's Theorem in the case of decaying potentials.  We state it in
terms of eigenfunctions: For almost every $x \in X$,
and for almost every $E$ in the essential support $\Lambda$ of the ac
spectrum, there is a pair of linearly independent complex conjugate
eigenfunctions $\uu(E,x,n)$ and $\overline {\uu(E,x,n)}$, normalized so that
the Wronskian
$\uu(E,x,n) \overline {\uu(E,x,n-1)}-\overline {\uu(E,x,n}) \uu(E,x,n-1)$ is
$i$, such that
\be \label {spectral}
\frac {1} {2 \pi} \int_\Lambda |\uu(E,x,n-1)|^2+|\uu(E,x,n)|^2 dE \leq 1,
\ee
with equality in the case of pure ac spectrum.\footnote {In general,
$\frac {1} {2 \pi} \int_\Lambda |\uu(E,x,n-1)|^2+|\uu(E,x,n)|^2 dE$ is
half the sum of the weights of the ac part of the spectral measures
associated to $\delta_n$ and $\delta_{n-1}$.}
Such an inequality is of course all that is needed to deduce a bound on the
average size of the transfer matrices:
\be \label {pi3}
\frac {1} {4 \pi}
\int_\Lambda \|A(E,m,n)\|+\|A(E,m,n)\|^{-1} dE \leq 1.
\ee

The ergodic setup has one advantage and one disadvantage with respect to the
decaying setup (as far as constructing counterexamples is concerned):
\begin{enumerate}
\item There is no need to ``spare potential'' in trying to promote
eigenfunction growth,
\item But potential we introduce must reappear (infinitely often), hence (by the
trend of products of large matrices to get larger) one risks promoting
``too much growth'', destroying the ac spectrum due to the need to
obey (\ref {pi3}).  (In other words, we have to ``spare ac spectrum''.)
\end{enumerate}

The effects of recurrence on transfer matrices growth is hard to neglect:
in particular, in the ergodic
case, eigenfunctions are known \cite {CJ}
to be everywhere bounded in any open interval
in the essential support of the ac spectrum (in the $\ell^2$ case,
the essential support is an interval, and one certainly can not hope for
boundedness everywhere).

In fact, it is quite difficult to achieve ac spectrum in the ergodic
context, which is of course what is behind the formulation of the
Kotani-Last Conjecture.  One of the known obstructions is Kotani's
Determinism Theorem \cite {K}, which can be stated as follows: If there
is some ac
spectrum, then the stochastic process $\{V(n)\}_{n \in \Z}$ is deterministic
in the sense that perfect knowledge of the past implies perfect knowledge of the
future.  The Kotani-Last Conjecture can then be seen as an optimistic
quantitative generalization of this result: almost periodicity means that
approximate knowledge of $\{V(n)\}_{n \in \Z}$
(i.e., up to $\ell^\infty$-small error) can be obtained by sufficiently
precise knowledge of the sequence in a sufficiently long
(but finite) interval.

\begin{rem}

Regarding the essential support of the
singular part of the spectral measure, it is well known that
bounded eigenfunctions can form at most a one-dimensional subspace.
The existence of bounded potentials
having no non-trivial bounded eigenfunctions was first
established by Jitomirskaya in \cite {J2} (for certain explicit
ergodic potentials with singular continuous spectrum).

\end{rem}

\subsection{Principles of construction}

As discussed above, a key obstacle to the construction of unbounded
eigenfunctions in the absolutely continuous spectrum is the need to obey
(\ref {pi3}).  In fact there are other similar constraints that must be
satisfied, for instance,
$\frac {1} {N} \sum_{k=0}^{N-1} \|A(E,0,k)\|$ is bounded for almost
every $E$.\footnote {Indeed if $\uu(E,x,n)$ and $\overline
{\uu(E,x,n)}$ is a pair of complex conjugate eigenfunctions with Wronskian
$i$ then $\frac {1} {\pi N} \sum_{0 \leq k \leq N-1} |\uu(x,E,n)|^2$ is
the derivative of the integrated density of states.}

The need for unbounded eigenfunctions to oscillate shows that
if we introduce growth, we must cancel it at some later scale. 
This demands very careful matching of the transfer matrices: if $x,y \in
\SL(2,\R)$ then we have $\|y x\| \geq \|y\| \|x\| |\sin \omega|$ where
$\omega$ is the angle between the contracting
eigendirection of $y^* y$ and the expanding eigendirection of $x x^*$.  So
unless the eigendirections are closely aligned, if $x$ and $y$ are large
then $y x$ is even larger.  But as energy
changes, the eigendirections move, which can easily destroy a precise
match, resulting in growth for nearby eigenfunctions, which will result in
losses of the ac spectrum.  Notice that (\ref {pi3}) shows also the need
to spread the moments where growth occurs according to the energy,
but this creates further complications regarding the interaction
of the transfer matrices in those different scales.

Our approach to avoid uncontrolled growth
is based on slow deformation of periodic potentials, the
spectrum of which consists of bands.
In order to create growth in the first
place, we must introduce disorder which eats up part of the ac spectrum.  In
order to lose only an $\epsilon$-proportion of ac spectrum, we must
introduce (in our approach) so little disorder that the corresponding growth
is of order $\epsilon$ in the bulk of the bands (this is clearly not enough
to win, due to the need to spare the ac spectrum).  However, we can
produce slightly more growth near the edges.  The disorder is introduced by
slow deformation, and then we unwind it.  The importance of slowness in the
deformation procedure is that any introduced eigenfunction growth is also
unwinded.  We get back to a bounded setting which allows us to iterate the
estimates.  What we see in the end is that an eigenfunction will tend to pick
up oscillation at some time scales.  While those oscillations are not
absolutely summable,
the process is so slow that their sum would still remain bounded unless
there is some coherence of the phases.
However, at rare random (and it is here one sees the spreading in the
energy) time scales the oscillations do become coherent, so
the eigenfunction does become unbounded.

We must of course be very careful in our introducing of disorder at each
step.  Our chosen mechanism is dictated by the setting.  In the continuum,
it is possible to introduce tiny amounts of rotation (for the transfer
matrices), and we proceed by a
large variation on the axis of rotation.  This does not work in the discrete
case, so we must instead create a tiny disturbance on the axis of rotation. 
In order to do so, we work all the time
with one-parameter families of periodic potentials that remain coherent in a
large part of the spectrum.

In order to construct non almost periodic potentials which have ac spectrum,
we consider again perturbations of periodic potentials.  We construct two
distinct deformations which are largely coherent, but which have slightly
distinct periods (in the continuum case).
Iterating each independently for a long time, they will
slowly lose the coherence, until it has a definite magnitude.  Later on they
will become coherent again, and we can match both to construct a new
periodic potential with large ac spectrum.  Geometrically, the dynamics has
slightly different speeds at nearby orbits of the phase space,
creating macroscopic sliding in long time scales (think of the horocycle
flow), though at some later time scale everything becomes periodic.  Sliding
is naturally incompatible with almost periodicity.  Technically, the
discrete case is much more delicate, since we can not produce a tiny
difference of periods (as it must be an integer), so we use slow
deformation along a coherent family of periodic potentials to construct
coherent periodic potentials with discrepant periods.

{\bf Acknowledgements:} I would like to thank David Damanik for numerous
comments.  This work was supported by the ERC Starting Grant
``Quasiperiodic'' and by the Balzan project of Jacob Palis.

%% file: continuum.tex
\section{Continuum case: preliminaries}

We will make use of the usual $\SL(2,\R)$ action on $\overline \C$:
$\begin{pmatrix} a&b\\c&d \end{pmatrix} \cdot z=\frac {az+b} {cz+d}$.

Let $d$ be the hyperbolic distance in the upper half-plane $\H$.

Let $R_\theta=\begin{pmatrix} \cos 2 \pi \theta & -\sin 2 \pi \theta \\ \sin
2 \pi \theta & \cos 2 \pi \theta \end{pmatrix}$.

If $A \in \SL(2,\R)$ satisfies $|\tr A|<2$, there exists a unique fixed
point $\u(A)$ of $A$ in $\H$.  Moreover, $A$ is conjugated in
$\SL(2,\R)$ to a well defined rotation $R_{\Theta(A)}$,
where $\Theta(A) \in (0,\frac {1} {2}) \cup (\frac {1} {2},1)$.
The conjugacy matrix $B$, satisfying $B A
B^{-1}=R_{\Theta(A)}$ is not canonical (one may postcompose $B$
with rotations), but can be chosen to have the form
\be
\B(A)=\frac {1} {(\Im \u(A))^{1/2}}
\begin{pmatrix} 1 & -\Re \u(A) \\ 0 & \Im \u(A) \end {pmatrix}.
\ee
Notice that $\u$ and $\B$ are analytic functions of $A$.

Given a continuous function $V:\R \to \R$,
we define the transfer matrices $A[V](E,t,s) \in \SL(2,\R)$
so that $A[V](E,t,t)=\id$ and
\be
\frac {d} {ds} A[V](E,t,s)=\begin{pmatrix} 0 & -E-V(s) \\ 1 & 0 \end
{pmatrix} A[V](E,t,s).
\ee
An eigenfunction of the Schr\"odinger operator with potential $V$ and energy
$E$ is a solution $\uu:\R \to \R^2$ of $\uu(s)=A[V](E,t,s) \cdot \uu(t)$.

We have the following basic monotonicity property:

\begin{lemma} \label {bla7}

If $s>t$ and $|\tr A[V](E,t,s)|<2$ then
\be
\frac {d} {dE} \Theta(A[V](E,t,s))>0.
\ee

\end{lemma}

\subsection{Periodic case}

Assume now that $V$ is periodic of period $T$.  In this case we write
$A[V](E,t)=A[V](E,t,t+T)$ and $A[V](E)=A[V](E,0)$.  Note that
$\tr A[V](E,t)=\tr A[V](E)$ for all $t \in \R$.

The spectrum $\Sigma=\Sigma(V)$
of the Schr\"odinger operator with potential $V$ is the set of all $E$ with
$|\tr A[V](E)| \leq 2$.  Let also
$\Omega=\Omega(V)$ be the set of all $E$ with $|\tr A[V](E)|<2$.  We note
that $\Sigma \setminus \Omega=\partial \Omega$ consists of isolated points.

For $E \in \Omega$, let $u[V](E,t)=\u(A[V](E,t))$ and
$u[V](E)=u[V](E,0)$.
The integrated density of states (i.d.s.) $N$
is absolutely continuous in this case.  It satisfies
\be
\frac {d} {dE} N(E)=
\frac {1} {2 \pi T} \int_0^T \frac {1} {\Im u[V](E,t)} dt,
\ee
for each $E \in \Omega$.

In the following results, we assume $V$ to be fixed and write $A(\cdot)$ for
$A[V](\cdot)$ and $u(\cdot)$ for $u[V](\cdot)$.

\begin{lemma} \label {bla3003}

For almost every $E \in \Sigma$, for every $\epsilon>0$,
there exists $N \in \N$ with the following property.
Let $\tilde u(E,t) \in \H$ be some (not necessarily periodic)
solution of $A(E,t,s) \cdot \tilde
u(E,t)=\tilde u(E,s)$.  Then
\be
\frac {1} {N T} \int_0^{N T} \frac {1} {\Im \tilde u(E,t)} dt>
\frac {1} {T} \int_0^T \frac {1} {\Im u(E,t)} dt-\epsilon.
\ee

\end{lemma}

\begin{proof}

Let $E \in \Omega$.  If $d(\tilde u(E,t),u(E,t))$
is large, then at least one of $\frac {1} {\Im \tilde u(E,t)}$ and
$\frac {1} {\Im \tilde u(E,t+T)}$ has to be large, so for
$N \geq 2$ we have
\be
\frac {1} {N} \sum_{j=0}^{N-1} \frac {1} {\tilde u(E,t+j T)} \geq \frac {1}
{\Im u(E,t)}.
\ee

Assume further that $\Theta(A(E))$ is irrational.
Then for any $0 \leq t \leq T$, as $N$ grows
the sequence $\tilde u(E,t+j T)$, $0 \leq j \leq N-1$, is getting
equidistributed in the circle of (hyperbolic radius)
$d(\tilde u(E,t),u(E,t))$ around
$u(E,t)$.  One directly computes that if $\tilde z, \tilde z' \in \H$
are symmetric points with respect to some $z \in \H$ then $\frac {1} {2}
(\frac {1} {\Im \tilde z}+\frac {1} {\Im \tilde z'})
\geq \frac {1} {\Im z}$.  It follows
that for $N$ large
\be
\frac {1} {N} \sum_{j=0}^{N-1} \frac {1} {\tilde u(E,t+j T)} \geq \frac {1}
{\Im u(E,t)}-\epsilon.
\ee
Since this estimate is uniform on the solution $\tilde u(E,t)$ provided
$d(\tilde u(E,t),u(E,t))$ is bounded, the result follows.
\end{proof}

\begin{lemma} \label {bla3}

For almost every $E \in \Sigma$, for every $t_0 \in \R$ we have
\be
\inf_{w \in \R^2,\, \|w\|=1}
\sup_{t>t_0} \|A(E,t_0,t) \cdot w\|=\sup_t
e^{(d(u(E,t),i)-d(u(E,t_0),i))/2}.
\ee

\end{lemma}

\begin{proof}

Take $E \in \Omega$ such that $\Theta(A(E))$ is irrational.
\end{proof}

\subsection{Uniquely ergodic case}


Let $\{F_t:X \to X\}_{t \in \R}$, be a continuous flow which is
minimal and uniquely ergodic with invariant probability
measure $\mmu$, and let $v:X \to \R$ be a
continuous function.  Let $A[F,v](E,x,t,s)=A[V](E,t,s)$ with
$V(t)=v(F_t(x))$
Let $\Sigma=\Sigma(F,v)$
be the corresponding spectrum (which is $x$-independent by minimality).
Notice that if $v$ is non-negative and non-identically vanishing,
then $\Sigma \subset (0,\infty)$.

Let $N(E)$ be the i.d.s., and $L(E)$ be the Lyapunov exponent, defined by
\be
L(E)=\lim_{T \to \infty}
\frac {1} {T} \int \ln \|A[F,v](E,x,0,T)\| d\mmu(x).
\ee
For a uniquely ergodic flow, the i.d.s. is not
necessarily absolutely continuous.  However, we have the following result
(due to Kotani, see \cite {D}). 
Let $\Sigma_0 \subset \Sigma$ be the set of $E$ with $L(E)=0$.

\begin{lemma} \label {bla3031}

We have $\frac {d} {dE} N(E)>0$ for almost every $E \in \Sigma_0$.
Moreover, there exists a measurable function
$u[F,v]:\Sigma_0 \times X \to \H$, unique up to
$\Leb \times \mmu$-zero
measure sets, such that $A[F,v](E,x,0,t) \cdot
u[F,v](E,x)=u[F,v](E,F_t(x))$.  This function satisfies
\be
\frac {d} {dE} N(E)=\frac {1} {2 \pi}
\int \frac {1} {\Im u[F,v](E,x)} d\mmu(x).
\ee

\end{lemma}

Notice that when $F$ is $T$-periodic, we have
$u[F,v](E,x)=u[V](E)$ for $V=v(F_t(x))$.

The following is due to Kotani, see \cite {D}.

\begin{thm} \label {3041}

Let $F_t:X \to X$ be a uniquely ergodic flow, and let $v:X \to \R$ be
continuous.
If the Lyapunov exponent vanishes in the spectrum and the ids is absolutely
continuous, then the spectral measures are absolutely continuous for almost
every $x \in X$.

\end{thm}

\subsection{Solenoidal flows}

If $K$ and $K'$ are compact Abelian groups, a projection $K \to K'$ is a
continuous surjective homomorphism.

Let $K$ be a totally disconnected compact
Abelian group, and let $i:\Z \to K$ be a
homomorphism with dense image.  The solenoid associated to $(K,i)$ is the
compact Abelian group obtained as the quotient of $K \times \R$ by the
subgroup $\{(i(-j),j)\}_{j \in \Z}$.  It comes with a canonical
projection $\pi:S \to \R/\Z$, $\pi(x,t)=t$.

Given $S$ as above, the solenoidal flow on $S$ is $F^S_t:S \to S$,
$F^S_t(x,s)=(x,t+s)$.

A time-change of $F^S$ is a flow $F_t:S \to S$ of the form
$F_t(x,s)=(x,h(x,s,t))$ with $t \mapsto h(x,s,t)$ $C^1$ for each $x$ and
$s$, and such that
$w_F(x,s)=\frac {d} {dt} h(x,s,t)|_{t=0}$ is a continuous positive
function of $(x,s)$.  Notice that any continuous positive function in $S$
generates a time-change.

Notice that a time change of a solenoidal flow is uniquely ergodic, and its
invariant probability measure is absolutely continuous with respect to the
Haar measure, with continuous positive density proportional to $\frac {1}
{w_F}$.

If $(K,i)$ and $(K',i')$ are as above, and there is a (necessarily unique)
projection $p_{K,K'}:K \to K'$ such that $p_{K,K'} \circ i=i'$, then we can
define a projection $p_{S,S'}:S \to S'$ in the natural way.

If $F$ and $F'$ are time-changes of $F^S$ and $F^{S'}$,
we say that $F$ is $\epsilon$-close to the lift of $F'$
if $|\ln w_{F'} \circ p_{S,S'}-\ln w| \leq \epsilon$.
We say that $v:S \to \R$
is $\epsilon$-close to the lift of $v':S' \to \R$ if
$|v' \circ p_{S,S'}-v| \leq \epsilon$.

In all cases we will deal with (e.g., $K=\Z/n \Z$),
there is a natural choice for the embedding
$i:\Z \to K$.  Thus we will often omit the embedding $i$ from the notation below.





\subsection{Projective limits}

An increasing sequence of compact Abelian groups is the data given by a
sequence $K_j$ of compact Abelian groups, and of projections
$p_{j',j}:K_{j'} \to K_j$, $j'>j$ such
that $p_{j_2,j_1} \circ p_{j_3,j_2}=p_{j_3,j_1}$.

Given such an increasing family of compact Abelian groups, there exists a
compact Abelian group $K$ and a sequence of projections $p_j:K
\to K_j$ such that $p_{j',j} \circ p_{j'}=p_j$ for every $j'>j$, and the
$p_j$ separate points in $K$: one takes $K$ as the set of infinite sequences
$x_j \in K_j$ with $p_{j',j}(x_{j'})=x_j$, endowed with the product topology
and obvious group structure.  We call $K$ the projective limit of the $K_j$.

When considering pairs $(K_j,i_j)$ as before,
we will assume moreover that the
projections are compatible with the embeddings, so that
$p_{j',j}=p_{K_{j'},K_j}$.

Notice that if the $K_j$ are totally disconnected then the projective limit
is totally disconnected.  Moreover, if $S_j$ is the solenoid over $K_j$, then
the projective limit $S$ of the $S_j$ is the solenoid over $K$.

An immediate application of projective limits is the following:

\begin{lemma} \label {bla3022}

Let $S_j$ be an increasing sequence of solenoids, and let $S$ be the
projective limit.
Let $F^j$ be time-changes of the solenoidal flows $F^{S_j}$.
Let $v_j:S_j \to \R$ be continuous functions.  Assume that for $j'>j$,
$(F^{j'},v_{j'})$ is $\epsilon_j$-close
to the lift of $(F^j,v_j)$,
where $\epsilon_j \to 0$.  Then there exists a time-change $F$ of the
solenoidal flow $F^S$, and a continuous function $v:S \to \R$ such that
$(F,v)$ is $\epsilon_j$-close to the lift of
$(F^j,v_j)$ for every $j$.

\end{lemma}

\begin{proof}

Define $S$ as a projective limit of the $S_j$ and take $v=\lim v_j \circ
p_j$, $w_F=\lim w_{F_j} \circ p_j$.
\end{proof}

\subsection{Lifting properties}

The following is a standard ``semi-continuity of the spectrum'' property:

\begin{lemma} \label {bla3014new}

Let $F'$ be a time-change of a solenoidal flow $F^{S'}$, and let $v':S'
\to \R$ be a continuous function.  Then for every $M>0$, $\epsilon>0$, there
exists $\kappa>0$ such that if either
$(F,v)$ is $\kappa$-close to the lift of $(F',v')$, or $(F',v')$ is
$\kappa$-close to the lift of $(F,v)$, then
$\Sigma(F,v) \cap (-\infty,M]$ is contained in the $\epsilon$-neighborhood
of $\Sigma(F',v')$, and $\Sigma(F',v') \cap (-\infty,M]$
is contained in the $\epsilon$-neighborhood of $\Sigma(F,v)$.

\end{lemma}

It easily implies:

\begin{lemma} \label {bla3014}

Let $F'$ be a time-change of a solenoidal flow $F^{S'}$, and let $v':S'
\to \R$ be a continuous function.  Then for every $M>0$, $\epsilon>0$, there
exists $\kappa>0$ such that if
$(F,v)$ is $\kappa$-close to the lift of $(F',v')$, then we have
$|\Sigma(F,v) \cap (-\infty,M] \setminus \Sigma(F',v')|<\epsilon$.

\end{lemma}

We say that $(F,v)$ is $(\epsilon_1,C_1,M)$-crooked if there
is a set $\Gamma \subset \Sigma \cap (-\infty,M]$ such that
$|(\Sigma \setminus \Gamma) \cap (-\infty,M]|<\epsilon_1$, and for
every $E \in \Gamma$, the set of $x \in X$ such that
\be
\inf_{w \in \R^2,\, \|w\|=1}
\sup_{t>0} \|A[F,v](E,x,0,t) \cdot w\|>C_1
\ee
has $\mmu$-measure (strictly) larger than $1-\epsilon_1$.

This notion provides a quantification of how large the
eigenfunctions are for most of the parameters.  Largeness can be checked in
many cases by bounding the $u(E,x)$, see Lemma \ref {bla3}.

The following is obvious, but convenient:

\begin{lemma} \label {bla3007}

Let $F'$ be a time-change of a solenoidal flow
$F^{S'}$, and let $v':S'
\to \R$ be a continuous function.  Assume that $(F',v')$ is
$(\epsilon_1,C_1,M)$-crooked.  Then there exists $\kappa>0$ such that if
$(F,v)$ is $\kappa$-close to the lift of $(F',v')$ then $(F,v)$ is
$(\epsilon_1,C_1,M)$-crooked.

\end{lemma}

\begin{proof}

Just use that $v \circ F_t$ is close to $v' \circ F'_t \circ p_{S,S'}$ for
$t$ (arbitrarily) bounded.
\end{proof}

We say that $(F,v)$ is $(\epsilon,M)$-good if
\be
\sup_{\Sigma \cap (-\infty,M]} L(E)<\epsilon.
\ee

It also trivially behaves well under lifts:

\begin{lemma} \label {bla5}

Let $F'$ be a time-change of a solenoidal flow $F^{S'}$, and let $v':S'
\to \R$ be a continuous function.  Assume that $(F',v')$ is
$(\epsilon,M)$-good.  Then there exists $\kappa>0$ such that if
$(F,v)$ is $\kappa$-close to the lift of $(F',v')$, then $(F,v)$ is
$(\epsilon,M)$-good.

\end{lemma}

We say that $(F,v)$ is $(\epsilon,M)$-nice if
\be
N(M)-\int_{-\infty}^M \frac {dN} {dE} dE<\epsilon.
\ee
Niceness provides a measure of how absolutely continuous the i.d.s. is.

The following deserves an argument.

\begin{lemma} \label {bla3015}

Let $F'$ be a time-change of a periodic solenoidal flow
$F^{S'}$,\footnote {This result still holds without assuming periodicity.}
and let $v':S' \to \R$ be a continuous function.  Assume that
$(F',v')$ is
$(\epsilon,M)$-nice.  Then there exists $\kappa>0$ with the following
property.  Assume that $(F,v)$ is $\kappa$-close to the lift of
$(F',v')$, the Lyapunov exponent for $(F,v)$ vanishes in the spectrum,
and $|(\Sigma(F',v') \setminus \Sigma(F,v)) \cap (-\infty,M]|<\kappa$.
Then $(F,v)$ is $(\epsilon,M)$-nice.

\end{lemma}

\begin{proof}

Let $N$, $N'$ be the integrated density of states for $(F,v)$,
$(F',v')$.  It is clear that $N(M)$ is close to $N'(M)$.
Using Lemma \ref {bla3003} and Lemma \ref {bla3031},
we see that for almost every
$E' \in \Sigma(F',v')$, for every
$\epsilon'>0$, there exists $\delta>0$ such that
for almost every $E \in \Sigma$ which is $\delta$-close to $E'$, if
$\kappa>0$ is sufficiently small, we have
\be
\frac {d} {dE} N(E)>\frac {d} {dE} N'(E')-\epsilon'.
\ee
Integrating over $\Sigma(F,v) \cap (-\infty,M]$, and using that the Lebesgue
measure of the spectrum is close, we get $\int_{-\infty}^M \frac {d} {dE}
N dE$ close to $\int_{-\infty}^M \frac {d} {dE'} N' dE'$.  The result
follows.
\end{proof}

\subsection{Slow deformation}

The following is the basic estimate of \cite {FK}.

\begin{lemma} \label {blainductive}

Let $J \subset \R$ be a closed interval,
and let $A:J \times \R/\Z \to \SL(2,\R)$ be a smooth function such that $|\tr
A(E,t)|<2$ for $(E,t) \in J \times \R/\Z$.  Let $B(E,t)=\B(A(E,t))$,
$\theta(E,t)=\Theta(A(E,t))$.
Then for every $m,k \in \N$, there exists $n(m) \in \N$ and
$C_{k,m}>0$
such that for every $n \geq n(m)$,
there exist smooth functions
$B_{(m,n)}:J \times \R/\Z \to \SL(2,\R)$, $\theta_{(m,n)}:J \times \R/\Z \to \R$
such that
\begin{enumerate}
\item $\|A_{(m,n)}-R_{\theta_{(m,n)}}\|_{C^k} \leq \frac {C_{k,m}} {n^m}$,
where
\be
A_{(m,n)}(E,t)=B_{(m,n)}(E,t+\frac {1} {n}) A(E,t) B_{(m,n)}(E,t)^{-1},
\ee
\item $\|B_{(m,n)}-B\|_{C^k} \leq \frac {C_{k,m}} {n}$,
\item $\|\theta_{(m,n)}-\theta\|_{C^k} \leq \frac {C_{k,m}} {n}$.
\end{enumerate}

\end{lemma}

\begin{proof}

Consider first the case $m=1$.  In this case, set $B_{(1,n)}=B$,
$\theta_{(1,n)}=\theta$, and the estimate is obvious.

Assume we have proved the result for $m \geq 1$.  Let
\be
B_{(m+1,n)}(E,t)=\B(A_{(m,n)}(E,t)) B_{(m,n)}(E,t),
\ee
\be
\theta_{(m+1,n)}(E,t)=\Theta(A_{(m,n)}(E,t)).
\ee
The estimates follow from the induction hypothesis.
\end{proof}

Under a monotonicity assumption, it yields a parameter
estimate:

\begin{lemma} \label {blaparameter}

Under the hypothesis of the previous lemma, assume moreover that
$\tilde \theta(E)=\int_{\R/\Z}\theta(E,t) dt$ satisfies
$\frac {d} {dE} \tilde \theta(E) \neq 0$ for every $E \in J$.
For $n \in \N$, let $A^{(n)}:J \times \R/\Z \to \SL(2,\R)$ be given by
\be
A^{(n)}(E,t)=A(E,t+\frac {n-1} {n}) A(E,t+\frac {n-2} {n}) \cdots
A(E,t+\frac {1} {n}) A(E,t).
\ee
Then there exist functions $\tilde \theta^{(n)}:J \to \R/\Z$ such that for
every measurable subset $Z \subset \R/\Z$,
\be
\lim_{n \to \infty} |\{E \in J,\, \tilde \theta^{(n)}(E) \in Z\}|=|Z| |J|,
\ee
with the following property.  For every $\delta>0$,
\be
\lim_{n \to \infty} \|\tr A^{(n)}(E,t)-2 \cos 2 \pi \tilde
\theta^{(n)}(E)\|_{C^0(J \times \R/\Z,\R)}=0,
\ee
\be
\lim_{n \to \infty} \sup_{|\sin 2 \pi \tilde \theta^{(n)}(E)|>\delta}
\|\Theta(A^{(n)}(E,\cdot))-
\tilde \theta^{(n)}(E)\|_{C^1(\R/\Z,\R)}=0,
\ee
\be
\lim_{n \to \infty} \sup_{|\sin 2 \pi \tilde \theta^{(n)}(E)|>\delta}
\|\u(A^{(n)}(E,\cdot))-\u(A(E,\cdot))\|_{C^1(\R/\Z,\C)}=0.
\ee

\end{lemma}

\begin{proof}

Let $B_{(m,n)}$, $A_{(m,n)}$ and $\theta_{(m,n)}$
be as in Lemma \ref {blainductive}, and let
\be
A^{(m,n)}(E,t)=B_{(m,n)}(E,t) A^{(n)}(E,t) B_{(m,n)}(E,t)^{-1}.
\ee
We have
\be
A^{(m,n)}(E,t)=A_{(m,n)}(E,t+\frac {n-1} {n}) \cdots A_{(m,n)}(E,t).
\ee
Let $\theta^{(m,n)}(E,t)=\sum_{j=0}^{n-1} \theta_{(m,n)}(E,t+\frac {j}
{n})$.  Then
\be
A^{(m,n)}-R_{\theta^{(m,n)}}=\sum_{j=1}^n H_{(m,n,j)},
\ee
with
\be
H_{(m,n,j)}=\sum H_{(m,n,\underline i)}
\ee
where the sum runs over all non-empty
sequences $\underline i=(i_1,...,i_j)$
satisfying $0 \leq i_1<...<i_j<n$, and $H_{(m,n,\underline i)}$ is a
product $T_{n-1} \cdots T_0$ with $T_l(E,t)=R_{\theta_{(m,n)}(E,t+\frac {l}
{n})}$ if $l \neq i_r$
for every $1 \leq r \leq j$,
and $T_l(E,t)=A_{(m,n)}(E,t+\frac {l} {n})-R_{\theta_{(m,n)}(E,t+\frac {l}
{n})}$ if $l=i_r$ for some $1 \leq r \leq j$.  Then
\be
\|H_{(m,n,\underline i)}\|_{C^0} \leq
\|A_{(m,n)}-R_{\theta_{(m,n)}}\|_{C^0}^j,
\ee
\begin{align}
\|D H_{(m,n,\underline i)}(E,t)\|_{C^0} \leq
&j \|D (A_{(m,n)}-R_{\theta_{(m,n)}})\|_{C^0}
\|A_{(m,n)}-R_{\theta_{(m,n)}}\|_{C^0}^{j-1}\\
\nonumber
&
+(n-j) \|D R_{\theta_{(m,n)}}\|_{C^0}
\|A_{(m,n)}-R_{\theta_{(m,n)}}\|_{C^0}^j,
\end{align}
where we write $D$ for the total derivative.
Thus
\be
\|H_{(m,n,\underline i)}\|_{C^1} \leq \frac {C_m^j} {n^{m j-1}},
\ee
\be
\|A^{(m,n)}-R_{\theta^{(m,n)}}\|_{C^1} \leq \sum_{j=1}^n \frac {C_m^j}
{n^{(m-1) j-1}},
\ee
so that for $m \geq 3$ we have
\be
\lim_{n \to \infty} \|A^{(m,n)}-R_{\theta^{(m,n)}}\|_{C^1}=0.
\ee
Note that
\be
\theta^{(m,n)}(E,t)=n \sum_{k \in n \Z} e^{2 \pi i k t}
\int_{\R/\Z} \theta_{(m,n)}(E,t) e^{-2 \pi i k t} dt.
\ee
Let $\tilde \theta_{(m,n)}(E)=\int_{\R/\Z} \theta_{(m,n)}(E,t) dt$.
Then for $m \geq 3$, using that $\theta_{(m,n)}$ is uniformly $C^3$ for
fixed $m$ and $n \to \infty$,
\be
\lim_{n \to \infty}
\sup_{E \in J} \|\theta^{(m,n)}(E,\cdot)-n \tilde
\theta_{(m,n)}(E)\|_{C^1(\R/\Z,\R)}=0.
\ee
Since
\be
\lim_{n \to \infty}
\|\tilde \theta_{(m,n)}(E)-\tilde \theta(E)\|_{C^1}=0,
\ee
and the derivative of $\tilde \theta(E)$ is non-vanishing,
it follows that for $m \geq 3$ and
each measurable set $Z \subset \R/\Z$,
\be
\lim_{n \to \infty} |\{E \in J,\, \tilde \theta^{(m,n)}(E) \in Z\}|=
|Z| |J|.
\ee

Fix $m \geq 3$ and let $\tilde \theta^{(n)}=\tilde \theta^{(m,n)}$.
Then for $n$ large we get
\be
\lim_{n \to \infty} \|\tr A^{(m,n)}(E,t)-2 \cos 2 \pi \tilde
\theta^{(n)}(E)\|_{C^0(J \times \R/\Z,\R)}=0,
\ee
\be
\lim_{n \to \infty} \sup_{|\sin 2 \pi \tilde \theta^{(n)}(E)|>\delta}
\|\u(A^{(m,n)}(E,\cdot))-i\|_{C^1(\R/\Z,\C)}=0,
\ee
\be
\lim_{n \to \infty} \sup_{|\sin 2 \pi \tilde \theta^{(n)}(E)|>\delta}
\|\Theta(A^{(m,n)}(E,\cdot))-
\tilde \theta^{(n)}\|_{C^1(\R/\Z,\SL(2,\R))}=0.
\ee
Notice that $\tr A^{(m,n)}=\tr A^{(n)}$ and
$\Theta(A^{(m,n)})=\Theta(A^{(n)})$.  Moreover,
$\u(A^{(n)}(E,t))=
B_{(m,n)}(E,t)^{-1} \u(A^{(m,n)}(E,t))$.  Since
\be
\lim_{n \to \infty} \|B_{(m,n)}-B\|_{C^1}=0,
\ee
and $B(E,t) \cdot \u(A(E,t))=i$, we conclude that
\be
\lim_{n \to \infty} \sup_{|\sin 2 \pi \tilde \theta^{(n)}(E)|>\delta}
\|\u(A^{(n)}(E,\cdot))-\u(A(E,\cdot))\|_{C^1(\R/\Z,\C)}=0.
\ee
\end{proof}

\comm{
Thus, if $m \geq k+2$, we get

By the Leibniz formula,
\be
B_{(m,n)}(E,t) A_{(n)}(E,t) B_{(m,n)}(E,t)^{-1}-
R_{\theta'_{(n)}(E,t)\|_{C^k} \leq \frac {C'_{k,m}} {n^{m-1}}.
\ee

Let $J=[E_0-\delta_0,E_0+\delta_0]$ sufficiently small so that $|\tr
A(E,t)|<2$ for every $E \in J$, $t \in \T$.
We can write $A(E,t)=B(E,t) R_{\phi(E,t)} B(E,t)^{-1}$ for
$|E-E_0|<\delta_0$.
where $\phi:J \times \R/\Z \to \R$ and $B:J \times
\R/\Z \to \SL(2,\R)$ are smooth functions.  Let
$A^{(1)}(E,t)=B(E,t+1/n)^{-1} A(E,t) B(E,t)$.
Then $\|A^{(1)}-R_\phi\|_{C^j} \leq
\frac {C_j} {n}$.  Thus for $n$ large we can write
$A^{(1)}(E,t)=B^{(1)}(E,t) R_{\phi^{(1)}(E,t)} B^{(1)}(E,t)^{-1}$, where
$\|\phi^{(1)}-\phi\|_{C^j} \leq
\frac {C_{j,1}} {n}$, $\|B^{(1)}(t)-\id\|_{C^j} \leq \frac {C_{j,1}} {n}$.

Iterating, we get for every $l \geq 2$, for every $n$ sufficiently large,
and for every $2 \leq m \leq l$, functions $A^{(m)}$, $B^{(m)}$,
$\phi^{(m)}$, such that
\be
A^{(m)}(E,t)=B^{(m-1)}(E,t) R_{\phi^{(m-1)}(E,t)}B^{(m-1)}(E,t),
\ee
\be
A^{(m)}(E,t)=B^{(m-1)}(E,t+1/n) R_{\phi^{(m-1)}(E,t)}B^{(m-1)}(E,t),
\ee
\be
\|\phi^{(m)}-\phi^{(m-1)}\|_{C^j} \leq
\frac {C_{j,m}} {n^m},
\ee
\be
\|B^{(m)}(t)-\id\|_{C^j} \leq \frac {C_{j,m}} {n^m}.
\ee
}
}

\section{Continuum case: unbounded eigenfunctions}

The potential we will produce will be a suitable projective limit of
periodic potentials.
The actual work we need to do is to inductively construct
suitable periodic potentials.




\subsection{Construction of periodic potentials}

Let $V:\R/T \Z \to \R$ be a continuous function with $V(0)=0$.
For $\delta>0$, $N \in \N$, $n \in \N$, the $(\delta,N,n)$-padding of
$V$ is the continuous function $V':\R/T' \Z \to
\R$, $T'=2 N n T+\delta \sum_{j=0}^{2n-1} \sin^{2 N} \pi \frac {j} {2 n}$,
given by the following conditions:
\begin{enumerate}
\item $V'(t)=V(t-a_j)$, $a_j \leq t \leq a_j+N T$, $0 \leq j \leq 2n-1$,
\item $V'(t)=0$, $a_j+N T \leq t \leq a_{j+1}$, $0 \leq j \leq 2n-1$,
\item $a_0=0$, $a_{j+1}=a_j+N T+\delta \sin^{2N} \pi \frac {j} {2 n}$.
\end{enumerate}
In other words, we repeat the periodic potential $N n$ times, but
with an extra ``padding'' with a small
interval of zeroes every $N$ repetitions.  The length of those intervals is
slowly modulated, but it is always small (at most $\delta$).

The goal of this section is to establish the following estimate:

\begin{lemma} \label {bla22}

Let $V^{(0)}:\R/T^{(0)} \Z \to \R$ be a smooth non-constant,
non-negative function with $V^{(0)}=0$ near $0$.  Then for every
$M,\xi>0$, there exists $C>0$ such that for every $C_0>0$,
for every $\delta>0$ sufficiently small,
there exist $0<P<\xi \delta^{-1}$,
and sequences $N^{(j)},n^{(j)}$, $1 \leq j \leq P$, such that if we define
$V^{(j)}:\R/T^{(j)} \Z \to \R$, $1 \leq j \leq P$ so that $V^{(j)}$ is
obtained by $(\delta,N^{(j)},n^{(j)})$-padding of $V^{(j-1)}$,
then there exists a compact subset $\Gamma \subset (0,M] \cap
\Omega(V^{(P)})$ such that
$|\Sigma(V^{(0)}) \cap (-\infty,M]
\setminus \Gamma|<\xi$ and for every $E \in \Gamma$ we have
\be
\sup_t d(u[V^{(P)}](E,t),i) \geq C_0,
\ee
\be \label {blaTP}
\frac {1} {T^{(P)}} \int_0^{T^{(P)}} d(u[V^{(P)}](E,t),i) dt \leq C.
\ee

\end{lemma}

We will need a few preliminary results.

\begin{lemma}

For every $C>0$, $M>0$, there exist $C'>0$ and $\delta_0>0$
with the following property.
Let $V:\R/T \Z \to \R$ be a smooth function with $V(t)=0$ near $0$,\footnote
{This neighborhood can be arbitrarily small, but this will influence the
constants below that depend on $v$.}
and let $A(\cdot)=A[V](\cdot)$ and $u(\cdot)=u[V](\cdot)$.
Let $E_0 \in \Omega(V) \cap [M^{-1},M]$.
Assume that $C^{-1}<d(u(E_0),E_0^{1/2} i)<C$.
Then there exists $\epsilon_0>0$ such that for every
$0<\epsilon<\epsilon_0$, for every $\kappa>0$,
for every $0<\delta<\delta_0$, for every $N$ sufficiently
large, for every $n$ sufficiently large, letting
$V'$ be the $(\delta,N,n)$-padding of $v$, $A'(\cdot)=A[V'](\cdot)$,
$u'(\cdot)=u[V'](\cdot)$, we have the following.  
There exists a compact set
$\Lambda \subset \Omega(V') \cap [E_0-\epsilon,E_0+\epsilon]$
such that
\begin{enumerate}
\item $|\Lambda|>2 (1-C' \delta) \epsilon$,
\item For $E \in \Lambda$, $d(u'(E),u(E))<\kappa$ and
$C^{-1}<d(u'(E),E^{1/2} i)<C$,
\item For $E \in \Lambda$,
\be
\sup_{t \in [0,T']} d(u'(E,t),i) \geq \sup_{t \in
[0,T]} d(u(E,t),i),
\ee
\item For any $C' \delta<\gamma<1/4$,
there exists a compact set $\Lambda' \subset
\Lambda$ with $|\Lambda'|>\gamma \epsilon$
such that for $E \in \Lambda'$,
\be
\sup_{t \in [0,T']} d(u'(E,t),i) \geq
\sup_{t \in [0,T]} d(u(E,t),i)+C'^{-1} \frac {\delta} {\gamma}.
\ee
\item For $E \in \Lambda$,
\be
\left |\frac {1} {T'} \int_0^{T'} d(u'(E,t),i) dt-
\frac {1} {T} \int_0^T d(u(E,t),i) dt \right |<\kappa,
\ee
\end{enumerate}

\end{lemma}

\begin{proof}

Let $D(E)=\begin {pmatrix} E^{1/4} & 0 \\ 0 & E^{-1/4} \end{pmatrix}$.
Let $G:\R_+ \times \R/\Z \to \SL(2,\R)$ be given by
$G(E,t)=D(E) R_{\delta
\frac {E^{1/2}} {2 \pi} \sin^{2N} \pi t}
D(E)^{-1} A(E)^N$.
We have
\be
A'(E)=G(E,\frac {2n-1} {2n}) G(E,\frac {2n-2} {2n})
\cdots G(E,\frac {1} {2n}) G(E,0).
\ee
%

We can write for $E$ near $E_0$,
\be
B(E) A(E) B(E)^{-1}=R_{\theta(E)},
\ee
where $B(E)=\B(A(E))$ and $\theta(E)=\Theta(A(E))$.
By Lemma \ref {bla7}, $\theta$ has non-zero derivative.


Thus we can write
\be
G(E,t)=D(E) R_{\delta
\frac {E^{1/2}} {2 \pi} \sin^{2N} \pi t} D(E)^{-1} B(E)^{-1} R_{N
\theta(E)} B(E).
\ee
Letting $Q(E)=B(E) D(E)$, we get
\be
\tr G(E,t)=\tr Q(E) R_{\delta
\frac {E^{1/2}} {2 \pi} \sin^{2N} \pi t}
Q(E)^{-1} R_{N \theta(E)}
\ee
Notice that $Q(E) \notin \SO(2,\R)$, since $Q(E) \cdot i \neq i$ (here we
use that $B(E)^{-1} \cdot i=u(E) \neq E^{1/2} i=D(E) \cdot i$ for $E$ near
$E_0$).  Thus we
can write $Q=R^{(1)} D^{(0)} R^{(2)}$, a product of rotation, diagonal
and rotation matrices, depending analytically on $E$.  Then
\be
\tr G(E,t)=\tr D^{(0)}(E) R_{\delta
\frac {E^{1/2}} {2 \pi} \sin^{2N} \pi t}
D^{(0)}(E)^{-1} R_{N \theta(E)}.
\ee
Write $D^{(0)}(E)=\begin{pmatrix} \lambda(E) & 0 \\
0 & \lambda(E)^{-1} \end {pmatrix}$.  We may assume that
$\lambda(E)>1$.  Then
\be
\lambda(E)=e^{d(u(E),E^{1/2} i)/2},
\ee
so that $\frac {1} {2 C}<\ln \lambda(E)<\frac {C} {2}$.
Then
\begin{align}
\tr G(E,t)=&2 \cos ((\delta E^{1/2}
\sin^{2 N} \pi t)+2 \pi N \theta(E))\\
\nonumber
&
-(\lambda(E)-\lambda(E)^{-1})^2 \sin (\delta
E^{1/2} \sin^{2 N} \pi t) \sin 2 \pi N \theta(E).
\end{align}
Thus
\be
|\tr G(E,t)-2 \cos ((\delta
E^{1/2} \sin^{2 N} \pi t)+2 \pi N \theta(E))| \leq
C_1 \delta \sin 2 \pi N \theta(E). 
\ee
We conclude that if $2 N \theta(E)$ is at distance at least
$C_2 \delta$ from $\Z$, then $|\tr G(E,t)|<2$.

We conclude that for $\epsilon$ sufficiently small, for $N$ sufficiently
large, the set of $E \in [E_0-\epsilon,E_0+\epsilon]$ such that $|\tr
G(E,t)| \geq 2$ for some $t$ has Lebesgue measure at most
$C_3 \delta \epsilon$.  By Lemma \ref {blaparameter}, for $n$ large we
will have $|\tr A'(E)|<2$ for a compact set
$\Lambda(\epsilon,\delta,N,n) \subset
[E_0-\epsilon,E_0+\epsilon]$ of Lebesgue measure at least
$2 (1-C_4 \delta) \epsilon$.  We may further assume that for $E \in
\Lambda(\epsilon,\delta,N,n)$, the sequence $\{j \theta(E)\}_{0 \leq j \leq
N-1}$ is $\frac {1} {100}$ dense $\mod 1$.
Thus for such $E$, and any $w \in \H$, and any $0 \leq t \leq T$,
there exists $0 \leq k \leq N-1$ such
that $d(A(E,0,t) A(E)^k \cdot w,i) \geq d(u(E,t),i)+\frac {1} {2} d(w,u(E))$. 
Taking $w=u'(E,a_m)$ for some $0 \leq m \leq 2n-1$ (where $a_j$ is as in the
definition of a $(\delta,N,n)$-padding), we get
\be
d(u'(E,t+k T+a_m,i) \geq d(u(E,t),i)+\frac {1} {2} d(u'(E,a_m),u(E)).
\ee
In particular,
\be
\sup_t d(u'(E,t),i) \geq \sup_t d(u(E,t),i)+\frac {1} {2} \max_{0 \leq m
\leq 2 n-1} d(u'(E,a_m),u(E)).
\ee

Lemma \ref {blaparameter}
shows that $u'(E,a_m)$ is near $\u(G(E,\frac {m}
{2 n}))$ for $n$ large.  In particular, $u'(E)$
is near $u(E)$, since $G(E,0)=A(E)$ and $u'(E,a_n)$ is near
$w(E)=\u(G(E,1/2))$.

We want to estimate the hyperbolic
distance between $w(E)$ and $u(E)$ in $\H$.  Let $w'(E)$ be the fixed point
of $D^{(0)}(E) R_{\delta
\frac {E^{1/2}} {2 \pi}}
D^{(0)}(E)^{-1} R_{N \theta(E)}$ in $\H$.  Then $w(E)=
B^{-1} R^{(1)} \cdot w'(E)$.  Since $u(E)=B^{-1} R^{(1)} \cdot i$,
it follows that
\be
d(w(E),u(E))=d(w'(E),i).
\ee
But $w'(E)$ is the solution $z \in \H$ of the equation $a z^2+b z+c=0$,
where
\be
a=\cos \delta E^{1/2} \sin 2 \pi N \theta(E)+
\lambda(E)^{-2} \sin \delta E^{1/2}
\cos 2 \pi N \theta(E),
\ee
\be
b=(\lambda(E)^2-\lambda(E)^{-2}) \sin \delta E^{1/2}
\sin 2 \pi N \theta(E),
\ee
\be
c=\cos \delta E^{1/2} \sin 2 \pi N \theta(E)+
\lambda(E)^2 \sin \delta E^{1/2}
\cos 2 \pi N \theta(E).
\ee
Then
\begin{align}
\Im w'(E)=&\left (
1+\frac {(\lambda(E)^2-\lambda(E)^{-2})
\sin \delta E^{1/2} \cos 2 \pi N \theta(E)}
{\cos \delta E^{1/2} \sin 2 \pi N \theta(E)+
\lambda(E)^{-2} 
\sin \delta E^{1/2} \cos 2 \pi N \theta(E)} \right .\\
\nonumber
&\left .-\frac {(\lambda(E)^2-\lambda(E)^{-2})^2
\sin^2 \delta E^{1/2} \sin^2 2 \pi N \theta(E)}
{4(\cos \delta E^{1/2} \sin 2 \pi N \theta(E)+
\lambda(E)^{-2} 
\sin \delta E^{1/2} \cos 2 \pi N \theta(E))^2}
\right )^{1/2}
\end{align}
Under the condition that $2 N \theta(E)$ is at distance at least
$C_2 \delta$ from $\Z$, we have
\be
\left |\frac {(\lambda(E)^2-\lambda(E)^{-2})^2
\sin^2 \delta E^{1/2} \sin^2 2 \pi N \theta(E)}
{4(\cos \delta E^{1/2} \sin 2 \pi N \theta(E)+
\lambda(E)^{-2} 
\sin \delta E^{1/2} \cos 2 \pi N \theta(E))^2}
\right | \leq C_5 \delta^2,
\ee
\be
|\frac {(\lambda(E)^2-\lambda(E)^{-2})
\sin \delta E^{1/2} \cos 2 \pi N \theta(E)}
{\cos \delta E^{1/2} \sin 2 \pi N \theta(E)+
\lambda(E)^{-2} 
\sin \delta E^{1/2} \cos 2 \pi N \theta(E)}| \geq C_6^{-1}
\delta \cot 2 \pi N \theta(E).
\ee
If $2 N \theta(E)$ is at distance $C_2 \delta<\gamma<1/4$ from $\Z$
then
\be
|\frac {(\lambda(E)^2-\lambda(E)^{-2})
\sin \delta E^{1/2} \cos 2 \pi N \theta(E)}
{\cos \delta E^{1/2} \sin 2 \pi N \theta(E)+
\lambda(E)^{-2} 
\sin \delta E^{1/2} \cos 2 \pi N \theta(E)}| \geq C_7^{-1}
\frac {\delta} {\gamma},
\ee
so that
\be
d(w'(E),i) \geq C_8^{-1} \frac {\delta} {\gamma}.
\ee

It follows that in this case
\be
\sup_t u'(E,t) \geq \sup_t u(E,t)+C_9^{-1} \frac {\delta} {\gamma}.
\ee

For $C_2 \delta<\gamma<1/4$, let $\Lambda'(\epsilon,\delta,N,n,\gamma)$
be the set of $E \in \Lambda(\epsilon,\delta,N,n)$ such
that $2 N \theta(E)$ is at distance at most $\gamma$ from $\Z$.  Since
$\theta$ has non-zero derivative, we have
$|\Lambda'(\epsilon,\delta,N,n,\gamma)| \geq \frac {3} {2}
\gamma \epsilon$, for $\epsilon$ small, $N$
sufficiently large and $n$ sufficiently large.

\comm{
By the slow variation estimate, if $E \in
\Lambda(\epsilon,\delta,N,n)$ and $n$ is large, then $u'(E,a_{[n/4]})$ is
near $w(E)$.

$\Prod_{k=0}^{[n/4]-1} C(E,\frac {k} {n})$ takes $u'(E)$ to $\tilde w(E)$
near $w(E)$.

Assume that $\psi(E)$ is irrational.  Then the sequence
$A(E)^j \cdot \tilde w(E)$ is dense in a circle centered on $u(E)$ with
hyperbolic radius $d(\tilde w(E),u(E))$, and the
speed it is becoming dense only depends on $E$ and an upper bound on
$d(\tilde w(E),u(E))$.

Thus for such $E$ and any $0 \leq t \leq T$, we can find bounded
(depending on $E$) $k(t)$ such that
\be
d(A(E,0,t) A(E)^{k(t)} \cdot \wilde w(E),i) \geq
u(E,t)+C_9^{-1} \frac {\delta} {\gamma}.
\ee
If $N$ is sufficiently large, this shows that
\be
d(u'(E,[n/4] N T+k(t) T+t),i) \geq d(u(E,t),i)+C_9^{-1} \frac {\delta}
{\gamma}.
\ee
A similar argument gives
\be
d(u'(E,[n/4] N T+k(t) T+t),i) \geq d(u(E,t),i)+C_9^{-1} \frac {\delta}
{\gamma}.
\ee
}

To conclude, let us show that if $E \in \Lambda(\epsilon,\delta,N,n)$ and
$2 N \theta(E)$ is $C_2 \delta$-away from $\Z$, then
\be
\frac {1} {T'} \int_0^{T'} d(u'(E,t),i) dt-\frac {1} {T} \int_0^T
d(u(E,t),i) dt
\ee
is small.  The formulas for $w'$ imply that
$u'(E,t)$ is at bounded hyperbolic distance from some
$u(E,t')$.  In fact, if $a_j \leq t \leq
a_{j+1}$ then $u'(E,t)$ is at bounded hyperbolic distance from
$u(E,t-a_j)$.  Moreover, if $a_j \leq t \leq a_j+T$, then
$A(E,0,t-a_j)^{-1}
u'(E,t)$ is near $\u(G(E,\frac {j} {2 n}))$.  If
$\frac {j} {2 n}$ is not close to $\frac {1} {2}$, the
estimates give that the fixed point of $G(E,\frac {j} {2 n})$
is close to $u(E)$, provided $N$ is large.  It follows that $u'(E,t)$ is
near $u(E,t-a_j)$.  The result follows.
\end{proof}

\begin{lemma} \label {bla21}

For every $C>0$, $M>0$, there exist $C'>0$ and $\delta_0>0$ with the
following
property.  Let $V:\R/T\Z \to \R$ be a smooth non-negative function with
$V(t)=0$ near $0$.
Let $\Xi \subset \Omega(V) \cap
[M^{-1},M]$ be a compact subset such
that $C^{-1}<d(u[V](E),E^{1/2} i)<C$ for every $E \in \Lambda$.
Then for every $\kappa>0$, $R \in \N$, for every $0<\delta<\delta_0$, for
every $N$ sufficiently large, for every $n$ sufficiently large, if $V':\R/T'
\Z \to \R$ is the $(\delta,N,n)$-padding of $v$, then there exists a compact
subset $\Xi' \subset \Xi \cap \Omega(V')$ such that
\begin{enumerate}
\item For $j \geq 0$,
the conditional probability that $E \in \Xi$ belongs to $\Xi'$,
given that $\frac {j} {R} \leq \sup_t d(u[V](E,t),i)<
\frac {j+1} {R}$ is at least $1-2 C'$,
\item For every $E \in \Xi'$, $d(u[V'](E),u[V](E))<\kappa$ and
$C^{-1}<d(u[V'](E),E^{1/2} i)<C$,
\item For every $E \in \Xi'$,
\be
\sup_t d(u[V'](E,t),i) \geq \sup_t d(u[V'](E,t),i),
\ee
\item For $j \geq 0$, and for every $C' \delta<\gamma<1/4$,
the conditional probability that $E \in \Xi$ belongs to $\Xi'$ and
\be
\sup_t d(u[V'](E,t),i)>\sup_t d(u[V](E,t),i)+C'^{-1} \frac
{\delta} {\gamma},
\ee
given that $\frac {j} {R} \leq \sup_t d(u[V](E,t),i)<\frac {j+1}
{R}$ is at least $\frac {\gamma} {3}$.
\item For every $E \in \Xi'$,
\be
\left |\frac {1} {T'} \int_0^{T'} d(u[V'](E,t),i) dt-\frac {1} {T}
\int_0^T d(u[V](E,t),i) dt \right |<\kappa.
\ee
\end{enumerate}

\end{lemma}

\begin{proof}

Follows from the previous lemma by a covering argument.  (Notice that the
statements about conditional probabilities are automatic for large $j$,
since $\sup_t d(u[V](E,t),E^{1/2} i)$ is bounded by compactness of $\Xi$.)
\end{proof}

\noindent{\it Proof of Lemma \ref {bla22}.}
Notice that by non-constancy of $V^{(0)}$, $u^{(0)}(E) \neq E^{1/2} i$ for
almost every $E \in \Sigma(V^{(0)})$.  Up
to increasing $M$, we can assume that $\inf \Sigma(V^{(0)})>M^{-1}$.
Then for sufficiently large $C>0$,
there exists a compact subset $\Xi^{(0)} \subset \Sigma(V^{(0)}) \cap
(-\infty,M] \cap \Omega(V^{(0)})$ such that
$|(\Sigma(V^{(0)}) \setminus \Xi^{(0)}) \cap
(-\infty,M]|<\frac {\xi} {2}$, and for every $E \in \Xi^{(0)}$ we have
\be
C^{-1}<d(u[V^{(0)}](E),E^{1/2} i)<C,
\ee
and
\be
\sup_t d(u[V^{(0)}](E,t),i)<\frac {C} {2}.
\ee
Let $\delta_0=\delta_0(C,M)$ and $C'=C'(C,M)$ be as in
Lemma \ref {bla21}.

Let $P$ be maximal so that $(1-2 C' \delta)^P>1-\frac {\xi} {2 M}$.
Choose very small $0<\delta<\delta_0$, choose $R \in \N$ very large (in
particular, much larger than $\delta^{-1}$), and
take $\kappa>0$ very small.  Define sequences $V^{(j)}$, $\Xi^{(j)}$, $1
\leq j \leq P$, so that $V^{(j)}$, $\Xi^{(j)}$
is obtained by applying Lemma \ref {bla21} to $V^{(j-1)}$, $\Xi^{(j-1)}$.
It follows that
\be
|\Sigma(V^{(0)}) \cap (-\infty,M] \setminus \Xi^{(P)}| \leq \frac {\xi}
{2}+M (1-(1-2 C' \delta)^P)<\xi.
\ee
It also follows that
\be
\frac {1} {T^{(P)}} \int_0^{T^{(P)}} d(u[V^{(P)}](E,t),i) dt \leq \frac {C} {2}
+\kappa P<C.
\ee

Let $Z_j$, $0 \leq j \leq P$, be random variables on $\Xi^{(0)}$ given as
follows.  If $E \notin \Xi^{(j)}$, let $Z_j=Z_{j-1}+1$.  Otherwise, let
$Z_j=\frac {j} {R}$, where $j$ is maximal with
$\sup_t d(u[V^{(j)}](E,t),i) \geq \frac {j} {R}$.

Let $L \subset \N$ be the set of all $l$ with
$4 \delta C'^{-1} R<l<C'^{-2} R$.
We have $Z_0 \geq 0$, and the conditional probability that
$Z_j-Z_{j-1} \geq \frac {l} {R}=C'^{-1} \frac {\delta} {\gamma}$,
given $Z_{j-1}$ is at least $\frac {\gamma} {3}=\frac {\delta R}
{3 l C'}$, provided $C' \delta<\gamma<1/4$, i.e.
$l \in L$.  Consider i.i.d.
random variables $W_j$, $1 \leq j \leq P$, taking only values of the form
$\frac {l} {R}$ with $l=0$ or $l \in L$, and such that
\be \label {blawj}
p(W_j \geq \frac {l} {R})=\frac {\delta R} {3 l C'}.
\ee
whenever $l \in L$.  Since $Z_0 \geq 0$ and $p(Z_j-Z_{j-1} \geq \frac {l}
{R}|Z_{j-1})
\geq p(W_j \geq \frac {l} {R})$ for every $l \in \Z$, we get
\be
p(Z^{(P)} \geq \frac {m} {R}) \geq p(\sum_{j=1}^P W_j \geq \frac {m} {R})
\ee
for every $m \in \Z$.

To conclude, it is enough to show that $p(\sum_{j=1}^P W_j<C_0)<\xi/2$,
provided $\delta$ is sufficiently small.

By (\ref {blawj}), for $C''<k<-C''-\ln \delta$,
we have $p(2^k \delta<W_j<2^{k+1} \delta)>C''^{-1} 2^{-k}$ (here $C''$ is an
appropriately large constant depending on $C'$,
and we are using that $R$ is
much larger than $\delta^{-1}$).  We also have
$P \geq C'''^{-1} \delta^{-1}$ for some constant $C'''$ depending on
$M$, $\xi$ and $C'$.  By the Law of Large Numbers, for each $D \in \N$, and
each $C''<k \leq D$, if $\delta$ is sufficiently small, then with
probability at least $1-\frac {\xi} {4 D}$, we will have
$2^k \delta<W_j<2^{k+1} \delta$ for a set of $1 \leq j \leq P$ of
cardinality at least $C''^{-1} 2^{-k-1} P \geq
C'''^{-1} C''^{-1} 2^{-k-1} \delta^{-1}$.  This implies that, with
probability at least $1-\frac {\xi} {4}$,
$\sum W_j \geq \frac {D-[C'']} {2 C''' C''}$. 
The result follows by taking $D \geq C''+2 C''' C'' C_0$.
\qed

\subsection{Passing to the limit}

\begin{lemma} \label {bla25}

Let $F_t:S \to S$ be a $T$-periodic time-change of a solenoidal flow, and let
$V:\R/T \Z \to \R$, $v:S \to \R$ be continuous functions such that
$V(t)=v(F_t(0))$ is smooth.  Assume
that $V(t)=0$ for $T-\epsilon_0 \leq t \leq T$ for some
$0<\epsilon_0<T$.
Then for $0<\delta<\epsilon_0$, and
for every $N,n \in \N$, the $(\delta,N,n)$-padding $V'$ of
$V$ has the form
$V'(t)=v'(F'_t(0))$, where $F'_t:S' \to S'$ is a $T'$-periodic
time-change of a solenoidal flow, $v':S' \to \R$ is continuous, and
$(F',v')$ is $\frac {\delta} {\epsilon_0}$-close to a lift
of $(F_t,v)$.

\end{lemma}

\begin{proof}

Let $U_0=\{F_t(0),\, T-\epsilon_0<t<T\}$.
We take $S'$ as the $N n$-cyclic cover of $S$.  Let $p:S' \to S$ be the
corresponding projection.  Let $U'_{0,j}$,
$0 \leq j \leq 2 n N-1$, be the connected components of $U'_0=p^{-1}(U_0)$,
labeled so that they are positively cyclically ordered and such that
the right boundary of $U'_{0,2 n N-1}$ is $0$.
Then there exists a continuous non-positive function $\rho:S' \to \R$
such that $\rho=0$ outside $U'_0$ and any $U'_{0,j}$ with $j$ not divisible
by $N$, and such that
\be
\int_{U'_{0,j}} \frac {1} {e^{\rho(x)} w_F(p(x))} dx=
\epsilon_0+\delta \sin^{2 N} \pi \frac {j} {N}
\ee
is equal to
if $j$ is divisible by
$N$.  Indeed, we can take $\|\rho\|_{C^0}$ arbitrarily close to
$\ln \frac {\epsilon_0+\delta} {\epsilon_0}$, and hence less than $\frac
{\delta} {\epsilon_0}$.

The result
then follows with $v'=v \circ p$, and $w_{F'}=e^\rho w_F \circ p$.
\end{proof}

\begin{lemma} \label {bla3005}

Let $F_t:S \to S$ be a periodic time-change of a solenoidal flow, and
let $v:S \to \R$ be continuous non-constant non-negative function
such that $t \mapsto v(F_t(0))$ is smooth and $v(F_t(0))=0$ for
$t$ near $0$.  Then for every $\epsilon_1,C_1,M,\kappa>0$,
there exists a periodic
time-change of a solenoidal flow $F'_t:S' \to S'$, and a continuous
non-constant non-negative function such that $t \mapsto v'(F'_t(0))$ is
smooth, $v'(F'_t(0))=0$ for $t$ near $0$, $(F',v')$ is $\kappa$-close to a
lift of $(F,v)$, and $(F',v')$ is $(\epsilon_1,C_1,M)$-crooked.

\end{lemma}

\begin{proof}

Let $\epsilon_0>0$ be such that $v(F_t(0))=0$ for
$T-\epsilon_0 \leq t \leq T$.
Apply Lemma \ref {bla22} to $V^{(0)}(t)=v(F_t(0))$, with parameters $C_0$ and
$\xi<\epsilon_0 \kappa$
to be specified below, to get $\Gamma$ and $P$, and then
Lemma \ref {bla25} $P$ times, to get $(F',v')$ with
$V^{(P)}(t)=v'(F'_t(0))$ such that $(F',v')$ is
$\frac {\xi} {\epsilon_0}$-close to a lift of $(F,v)$.
By Lemma \ref {bla3014}, if $\xi$ is small then $|\Sigma(V^{(P)})
\cap (-\infty,M] \setminus \Sigma(V^{(0)})|<\epsilon_1/2$, and if
additionally $0<\xi<\epsilon_1/2$ we conclude that
$|\Sigma(V^{(P)}) \cap (-\infty,M] \setminus \Gamma|<\epsilon_1$.  Fix such
$\xi$ and let $C$ be as in Lemma \ref {bla22}.

By (\ref {blaTP}),
for every $E \in \Gamma$, the set $Z_E$ of all $t \in \R/T^{(P)}\Z$ with
\be
d(u[V^{(P)}](E,t),i) \leq \frac {C} {\epsilon_1}
\ee
has measure at least
$(1-\epsilon_1) T^{(P)}$.  By Lemma \ref {bla3}, for almost every $E \in
\Gamma$, $t_0 \in Z_E$ implies that
\be
\inf_{w \in \R^2,\, \|w\|=1} \sup_{t>t_0} \|A[V^{(P)}](E,t_0,t)\| \geq
\left (\frac {C_0 \epsilon_1} {C} \right )^{1/2},
\ee
So by taking $C_0=C_1^2 C/\epsilon_1$, we get that $(F',v')$ is
$(\epsilon_1,C_1,M)$-crooked.
\end{proof}

\noindent{\it Proof of Theorem \ref {continuumunbounded}.}
Let $V^{(0)}:\R/\Z \to \R$ be a smooth non-constant
non-negative periodic function with $V^{(0)}(t)=0$ near $0$.
We can see $\R/\Z$ as the solenoid $S^{(0)}$ corresponding to the trivial
group.  Let $F^{(0)}$ be the solenoidal flow on $S^{(0)}$.

Let $\kappa_0=1$.
Apply Lemma \ref {bla3005} inductively to obtain a sequence
$(F^{(j)},v^{(j)})$, $j \geq 1$ such that
\begin{enumerate}
\item $(F^{(j)},v^{(j)})$ is $(2^{-j},2^j,2^j)$-crooked,
\item $|(\Sigma(F^{(j)},v^{(j)}) \setminus
\Sigma(F^{(j-1)},v^{(j-1})) \cap (-\infty,2^j]|<\kappa_j$,
\item $(F^{(j)},v^{(j)})$ is $\kappa_j$-close to the lift of
$(F^{(j-1)},v^{(j-1)})$,
\item $\kappa_j<\kappa_{j-1}/2$ is chosen so small that if $(F,v)$ is
$2 \kappa_j$-close to the lift of
$(F^{(j-1)},v^{(j-1)})$, and
$|(\Sigma(F,v) \setminus
\Sigma(F^{(j-1)},v^{(j-1)}) \cap (-\infty,2^j]|<2 \kappa_j$,
then $(F,v)$ is
$(2^{-j},2^j)$-nice (use Lemma \ref {bla3015}),
$(2^{-j},2^j)$-good (use Lemma \ref {bla5}), and if $j \geq 2$,
$(2^{1-j},2^{j-1},2^{j-1})$-crooked (use Lemma \ref {bla3007}).
\end{enumerate}

Let $(F,v)$ be the limit of the $(F^{(j)},v^{(j)})$ obtained by Lemma \ref
{bla3022}.  Then $(F,v)$ is $(2^{-j},2^j)$-good for all
$j$, so the Lyapunov exponent vanishes in the spectrum.  Moreover,
$(F,v)$ is $(2^{-j},2^j)$-nice for all $j$,
so the i.d.s. is absolutely continuous.\footnote {Notice that since
$\Sigma(F,v)$ contains $\bigcap_{j' \geq j}
\Sigma(F^{(j')},v^{(j')})$ (see Lemma \ref {bla3014new}),
we must have $|(\Sigma(F,v) \setminus
\Sigma(F^{(j-1)},v^{(j-1)}) \setminus (-\infty,2^j]|<2 \kappa_j$.}
By Theorem \ref {3041}, for almost every $x \in S$ the spectral measure
is purely
absolutely continuous.  By construction, $(F,v)$ is
$(2^{-j},2^j,2^j)$-crooked for all $j$.  Thus for almost every $E$ in the
spectrum, for almost every $x \in S$, all non-trivial eigenfunctions are
unbounded.
\qed

\section{Continuum case: breaking almost periodicity}

The example discussed in the previous section can be verified to be not
almost periodic.  Here we will discuss a simpler example that will be easier
to analyze.

\subsection{Spectral measure}

Given a bounded continuous function $V:\R \to \R$, we denote by $\mu_V$ the
spectral measure for the line Schr\"odinger operator.  It has some basic
continuity property:

\begin{lemma} \label {blacont}

Let $V:\R \to \R$ be a bounded continuous function, and let $V^{(n)}:\R \to
\R$, $n \in \N$ be a sequence of uniformly bounded continuous functions
such that $V^{(n)} \to V$ uniformly on compact subsets of $\R$.  Then
$\int \phi d\mu_{V^{(n)}} \to \int \phi d\mu_V$
for every compactly supported
continuous function $\phi:\R \to \R$.

\end{lemma}

We will only need explicit formulas for the spectral measure in the case of
periodic potentials.  Let $V:\R/T\Z \to \R$ be continuous, and denote its
shifts by $V_s:t \mapsto V(s+t)$.  Then $\mu_{V_s}$
is absolutely continuous and
\be
\frac {d} {dE} \mu_{V_s}=\frac {1} {\Im u[V](E,s)}.
\ee

For $C>0$, let $\mu_{V_s,C}$ be the restriction of $\mu_{V_s}$
to the set of $E$
with $|\frac {d} {dE} \mu_{V_s}|<C$.  We say that a periodic $v$ is
$(\epsilon,C,M)$-uniform if
$\mu_{V_s}(-\infty,M]-\mu_{V_s,C}(-\infty,M]<\epsilon$ for every $s$.

We clarly have:

\begin{lemma} \label {blaunif1}

For every periodic $V$, $\epsilon>0$, $M>0$, there exists $C>0$ such that
$V$ is $(\epsilon,C,M)$-uniform.

\end{lemma}

\subsection{Weak mixing}

Let $F'_t:S' \to S'$ be a time-change of a periodic solenoidal flow.
We say that a time-change $F_t:S \to S$ of a solenoidal flow is
$(N,F')$-mixed if
$S$ projects onto $S'$ (through $p$), and for every
$1 \leq j \leq N$, there exists $t_j>0$ and compact subsets
$U_j,V_j \subset S$ with Haar measure
strictly larger than $1/3$, such that for each $x \in U_j$ there exists
$|t|<\frac {1}
{N}$ such that $p(F_{t_j}(x))=F'_t(p(x))$, and for each $x \in V_j$, there
exists $|t-\frac {j} {N}|<\frac {1} {N}$ such that
$p(F_{t_j}(x))=F'_t(p(x))$.

\begin{lemma} \label {blamixed1}

Let $F$ be $(N,F')$-mixed.  Then there exists $\kappa>0$ such
that if $\tilde F$ is $\kappa$-close to the lift of
$F$ then $\tilde F$ if $(N,F')$ mixed.

\end{lemma}

\begin{lemma} \label{blamixed2}

Let $F$ be the projective limit of $F^{(n)}$, and assume that for every
$N \in \N$, for every $n$ sufficiently large, $F$ is
$(N,F^{(n)})$-mixed.  Then $F$ is weak mixing.

\end{lemma}

\begin{proof}

If $F$ is not weak mixing, then there exists a non-trivial
eigenfunction, i.e., a
measurable function $\psi:S \to S^1$ such that $\psi \circ F_t=e^{2 \pi i
\theta t} \psi$ for some $\theta \in \R \setminus \{0\}$.  Let
$\psi^{(n)}:S^{(n)} \to \C$ be the expected value of $\psi$ on
$p_{S,S^{(n)}}^{-1}(x)$ (with respect to the Haar measure on $S$).  Then
$\psi^{(n)} \circ p_{S,S^{(n)}}$ converges to $\psi$ almost everywhere
(Martingale Convergence Theorem).

Since $\psi$ is an eigenfunction,
$t \mapsto \psi^{(n)}(F^{(n)}_t(0))$ is continuous, uniformly on $t$
and $n$.

By the definition of projective limit,
\be
\lim_{n \to \infty} \sup_{x \in S^{(n)}} \sup_{0 \leq t \leq 1}
|\psi^{(n)}(F^{(n)}_t(x)-e^{2 \pi i \theta t} \psi^{(n)}(x)|=0.
\ee

Thus, for $x \in U_j$,
$\psi^{(n)}(p_{S,S^{(n)}}(F_{t_j}(x)))$ is close to
$\psi^{(n)}(p_{S,S^{(n)}}(x))$.  Since the first is close to
$\psi(F_{t_j}(x))$ and the second is close to $\psi(x)$ for most
$x \in U_j$, this shows that $\theta t_j$ is close to an integer.

A similar argument using $V_j$, shows that
$\theta (t_j-\frac {j} {N})$ is close to an integer, so that
$\theta \frac {j} {N}$ is close to an integer.
Since $1 \leq j \leq N$ is arbitrary, we conclude that $\theta=0$.
\end{proof}

\subsection{The construction}

Let $V:\R/T \Z \to \R$ be a continuous function with $V(0)=0$.
For $\delta>0$, $n \in \N$, the $(\delta,n)$-padding (a simplified version
of a $(\delta,N,n)$-padding) of
$V$ is the continuous function $V':\R/T' \Z \to
\R$, $T'=2 n T+\delta n$, given by the following conditions:
\begin{enumerate}
\item $V'(t)=V(t-a_j)$, $a_j \leq t \leq a_j+T$, $0 \leq j \leq 2n-1$,
\item $V'(t)=0$, $a_j+T \leq t \leq a_{j+1}$, $0 \leq j \leq 2n-1$,
\item $a_j=j T$, $0 \leq j \leq n$, $a_j=j T+(j-n) \delta$, $n+1 \leq j \leq
2n$.
\end{enumerate}

\begin{lemma} \label {blamixed}

Let $F,v,V,\epsilon_0$ be as in Lemma \ref {bla25}.
Then for every $\delta>0$ sufficiently small, for every $N \in \N$,
for every $n$ sufficiently large,
the $(\delta,n)$-padding $V'$ of $V$ has the form
$V'(t)=v'(F'_t(0))$, where $F'_t:S' \to S'$ is a $T'$-periodic
time-change of a solenoidal flow, $v':S' \to \R$ is continuous,
$(F',v')$ is $\frac {\delta} {\epsilon_0}$-close to a lift
of $(F,v)$, and $(F',v')$ is $(N,F)$-mixed.

\end{lemma}

\begin{proof}

Let $N_S \in \N$ be the period of the solenoidal flow $F^S_t$.
Define $S'$ as the $2n$-cover of $S$.
Define a continuous function $\rho:S \to \R$ supported on $\{F_t(0),\,
T-\epsilon_0<t<T\}$ such that
$\int_0^T e^{-\rho(F_t(0))} dt=T+\delta$.  As in Lemma \ref {bla25}, we
can choose $\rho$ with $\|\rho\|_{C^0}<\frac {\delta} {\epsilon_0}$.
Let $\rho':S' \to \R$ be defined so that $\rho'=0$ on $[0,n N_S]$
and $\rho'=\rho \circ p_{S',S}$ on $[n N_S,2 n N_S]$.

Let $F'$ be the solenoidal flow with $w_{F'}=e^\rho w \circ p_{S',S}$, and
let $v'=v \circ p_{S',S}$.  All properties, but the last one, follow as in
Lemma \ref {bla25}.  For the last property, notice that if
$t_j=[\frac {j} {\delta
N}] (T+\delta)$ then for $x \in \{F'_s(0),\, 0 \leq s \leq n T-t_j\}$ we have
$p_{S',S} \circ F'_{t_j}(x)=F_{[j/(\delta N)] \delta}(p_{S',S}(x))$,
which belongs to $\{F_s(p_{S',S}(x)),\, \frac {j} {N}-\delta
\leq s \leq \frac {j}
{N}\}$, and for $x \in \{F'_s(0),\, n T \leq s \leq T'-t_j\}$ we have
$p_{S',S} \circ F'_{t_j}(x)=p_{S',S}(x)$.
\end{proof}

\begin{lemma} \label {blaunif}

Let $V:\R/T \Z$ be a continuous function with $V(0)=0$.  If $V$ is
$(\epsilon,C,M)$-uniform then for $\delta>0$
sufficiently small, for every $n \in \N$, if
$V'$ is the $(\delta,n)$-padding of $V$, then $V'$ is
$(\epsilon,C,M)$-uniform.

\end{lemma}

\begin{proof}

Let $A(\cdot)=A[V](\cdot)$.  Let $J \subset \Omega(V) \cap (-\infty,M]$
be a finite union of closed intervals such that
\be \label {shif}
\sup_s \mu_{V_s}(-\infty,M]-\mu_{V_s,C}(J)<\epsilon_0<\epsilon,
\ee
where
$V_s$ is the shift of $V$ and $\mu_{V_s,C}$ is the truncation of the
spectral measure.

If $E \in \Omega(V)$,
then $B(E) A(E) B(E)^{-1}=R_{\theta(E)}$, where $B=\B(A(E))$ and
$\theta(E)=\Theta(A(E))$ are analytic
functions and $\frac {d} {dE} \theta(E)>0$.  Let
\be
A_\delta(E)=
D(E) R_{\delta \frac {E^{1/2}} {2 \pi}} D(E)^{-1} A(E),
\ee
where $D(E)=\begin{pmatrix} E^{1/4} & 0 \\ 0 & E^{-1/4} \end{pmatrix}$.
Then
\be
A'(E)=A_\delta(E)^n A(E)^n,
\ee
where $A'(\cdot)=A[V'](\cdot)$.

For every $\kappa>0$, for $\delta>0$ sufficiently small, it is clear that for
every $0 \leq t \leq \delta$, for every $E \in J$,
we have $d(H(t) \cdot u(E),u(E))<\kappa$, where $H(t)=D(E)R_{t
\frac{E^{1/2}}
{2 \pi}} D(E)^{-1}$ is the exponential of
$t \begin{pmatrix} 0&-E\\1&0 \end{pmatrix}$ and $u(\cdot)=u[V](\cdot)$.

Notice that for $\delta>0$ sufficiently small, we have $|\tr A_\delta(E)|<2$
for $E \in J$.  Moreover, $B_\delta(E)=\B(A_\delta(E))$ and
$\theta_\delta(E)=\Theta(A_\delta(E))$ converge to
$B(E)$ and $\theta(E)$, when $\delta \to 0$,
as analytic functions of $E \in J$.  In particular,
\be
\lim_{\delta \to 0} \sup_n \sup_{E \in J}
\|B(E) A'(E) B(E)^{-1}-R_{n (\theta(E)+\theta_\delta(E))}\|=0.
\ee
For $0<\eta<1/2$,
let $J_{\delta,n,\eta} \subset J$ be the set of all $E$ such
that $2 n (\theta(E)+\theta_\delta(E))$ is at distance at least $\eta$ from
$\Z$.  Since $\frac {d} {dE}
\theta(E)>0$ and $\frac {d} {dE} \theta_\delta(E)>0$, we get, for every
$\delta>0$ small,
\be
\lim_{n \to \infty} |J_{\delta,n,\eta}|=(1-2 \eta) |J|.
\ee

For every $0<\eta<1/2$ and $\kappa>0$, if $\delta$
is sufficiently small, then
for every $n$ and for every $E \in J_{\delta,n,\eta}$, we have
$E \in \Omega(V')$ and $d(u'(E),u_\delta(E))+d(u_\delta(E),u(E))<\kappa$,
where $u'(\cdot)=u[V'](\cdot)$ and $u_\delta(E)=\u(A_\delta(E))$.
Notice that $A(E)^j \cdot u'(E)=u'(E,a_j)$ for $0
\leq j \leq n$, and $A_\delta(E)^{2n-j} \cdot u'(E,a_j)=u'(E)$ for $n \leq j
\leq 2n$.  In particular,
\be
d(u'(E,a_j),u(E))<\kappa.
\ee
Thus for $0 \leq j \leq 2n-1$ and $a_j \leq t \leq a_j+T$, we get
\be
d(u'(E,t),u(E,t-a_j))<\kappa.
\ee
For $n \leq j \leq 2n-1$ and $a_j+T \leq t \leq a_j+T+\delta$, we have
$H(a_j+T+\delta-t) u'(E,t)=u'(E,a_{j+1})$, so that
\begin{align}
d(&u'(E,t),u(E))=d(u'(E,a_{j+1}),H(a_j+T+\delta-t) \cdot u(E))\\
\nonumber
&
\leq d(u'(E,a_{j+1}),u(E))+d(H(a_j+T+\delta-t) \cdot u(E)),u(E))<2 \kappa.
\end{align}
It follows that for each $0 \leq t' \leq T'$ we can find some $0 \leq t(t') \leq
T$, defined by $t(t')=t'-a_j$ if $a_j \leq t' \leq a_j+T$ for some $j$,
and $t(t')=0$, if $a_j+T \leq t \leq a_{j+1}$ for some $j$,
such that for every $E \in J_{\delta,n,\eta}$, we have
$d(u'(E,t'),u(E,t))<2 \kappa$.

It follows that for every $0<\eta<1/2$ and $\kappa>0$, for
$\delta>0$ sufficiently small, for every $n$ sufficiently large,
\begin{align}
\mu_{V'_{t'},C}(J_{\delta,n,\eta}) &\geq e^{-2 \kappa}
\mu_{V_t,C}(J_{\delta,n,\eta})\\
\nonumber
&
\geq \mu_{V_t,C}(J)-(1-e^{-2 \kappa}) C |J|-
e^{-2 \kappa} C |J \setminus J_{\delta,n,\eta}|\\
\nonumber
&\geq \mu_{V_t,C}(J)-2
(\kappa+\eta) C |J|
\end{align}
(where $\mu_{V'_{t'},C}$ denotes the truncation of the
spectral measure $\mu_{V'_{t'}}$ for the shift of $V'$).
Thus, if $\eta+\kappa$ is sufficiently small, we get
\be
\mu_{V'_{t'},C}(-\infty,M] \geq \mu_{V_t,C}(J)-\frac {\epsilon-\epsilon_0} {2}.
\ee

Notice that $t(t')$ is such that for every
$C_1>0$, $\epsilon_1>0$, we have, for every $\delta>0$ sufficiently small,
for every $n \in \N$,
\be
\sup_{|s| \leq C_1} |V'(t'+s)-V(t+s)|<\epsilon_1.
\ee
By Lemma \ref {blacont}, if $\delta>0$ is sufficiently small we have
\be
\mu_{V_{t'}}(-\infty,M]<\mu_{V_t}(-\infty,M)+\frac {\epsilon-\epsilon_0} {2}.
\ee

Together with (\ref {shif}), it follows that
\be
\mu_{V'_{t'}}(-\infty,M]<\mu_{V'_{t'},C}(-\infty,M]+\epsilon,
\ee
as desired.
\end{proof}

\begin{rem} \label {blaremark3}

The construction also gives that for every $\kappa>0$,
for a subset of $\Omega(V)$ whose
complement has arbitrarily small measure, we have
\be
\sup_{t'} d(u[V'](E,t'),i) \leq \sup_t d(u[V](E,t),i)+2 \kappa.
\ee

\end{rem}

\noindent{\it Proof of Theorem \ref {continuumwm}.}
Define a sequence of $T^{(n)}$-periodic time-changes of solenoidal flows
$F^{(n)}_t:S^{(n)} \to S^{(n)}$ and a sequence of continuous
functions $v^{(n)}:S^{(n)} \to \R$ in the following way.

First take $T^{(0)}=1$, $S^{(0)}=\R/\Z$, $F^{(0)}_t=F^{S^{(0)}}_t$,
and $v^{(0)}:\R/\Z \to \R$ a non-constant smooth function with
$v^{(0)}=0$ near $0$.  Let $\kappa_0=1$.  Then for $j \geq 1$,
\begin{enumerate}
\item Choose $C_{j-1}>0$ so that
$t \mapsto v^{(j-1)}(F^{(j-1)}_t(0))$
is $(2^{1-j},C_{j-1},2^{j-1})$-uniform,
\item Choose $(F^{(j)},v^{(j)})$ so that it is
$(2^{-j'},C_{j'},2^{j'})$-uniform for all $0 \leq j' \leq j-1$,
$F^{(j)}$ is $(2^{j-1},F^{(j-1)})$-mixed, and $(F^{(j)},v^{(j)})$ is
$\kappa_{j-1}$-close to a lift of $(F^{(j-1)},v^{(j-1)})$
\item Let $0<\kappa_j<\kappa_{j-1}/2$
be such that if $(F,v)$ is $2 \kappa_j$-close to the lift of
$(F^{(j)},v^{(j)})$ then $(F,v)$ is $(2^{j-1},F^{(j-1)})$-mixed.
\end{enumerate}
The first step is an application of Lemma \ref {blaunif1}, the second
is an application of Lemmas \ref {blaunif} (notice that
by the previous choices, $(F^{(j-1)},v^{(j-1)}$ is
$(2^{-j'},C_{j'},2^{j'})$-uniform for all $0 \leq j' \leq j-1$)
and \ref {blamixed}, and the third is
an application of Lemma \ref {blamixed1}.

Let $S$ be the projective limit of the $S^{(j)}$ and let
$(F,v)$ be the projective limit of the $(F^{(j)},v^{(j)})$.  Then
$F$ is $(2^j,F^{(j)})$-mixed for all $j \geq 1$,
so it is weak mixing by Lemma
\ref {blamixed2}.  We also have that for every $x \in S$,
$V^{(j)}_x:t \mapsto
v^{(j)}(F^{(j)}_t(p_{S,S^{(j)}}(x))$ converges to
$V_x:t \mapsto v(F_t(x))$
uniformly on compacts.  It follows that the spectral measure
$\mu=\mu_{V_x}$
is the limit of the spectral measures $\mu_{V^{(j)}_x}$. 
For every $C>0$, and up to taking a subsequence,
the truncations $\mu_{V^{(j)}_x,C}$ converge to 
a measure $\mu_C \leq \mu$
which is absolutely continuous with density bounded by $C$. 

Then we have
\be
\mu(-\infty,2^j)-\int_{-\infty}^{2^j} \frac {d \mu} {dE} dE
\leq \lim_{k \to
\infty} \mu_{V^{(k)}_x}(-\infty,2^j)-
\mu_{V^{(k)}_x,C_j}(-\infty,2^j) \leq  2^{-j}.
\ee
The result follows.
\qed

\begin{rem} \label {blaremark4}

Notice that the construction allows us to obtain
(by Remark \ref {blaremark3}), that for $k \in \N$ there exists a subset
$\Gamma^{(k)} \subset \Omega(V^{(k)}) \cap \Omega(V^{(k+1)})$ such that
$|\Omega^{(k)} \setminus \Gamma^{(k)}| \leq 2^{-k}$ and
\be
\sup_t d(u[V^{(k+1)}](E,t),i) \leq \sup_t d(u[V^{(k)}](E,t),i)+2^{-k}.
\ee
Moreover, by Lemma \ref {bla3014}, we may also assume that
$|\Sigma(F,v)
\setminus \Omega(V^{(k)})| \leq 2^{-k}$.
It follows that for almost every $E \in
\Sigma(F,v)$, there exists $C(E)>0$ such that
$E \in \Omega(V^{(k)})$ for every $k$ sufficiently large and
$\sup_t d(u[V^{(k)}](E,t),i) \leq C(E)$.  This implies that
$\sup_t \sup_s \|A[V^{(k)}](E,t,s)\| \leq e^{C(E)}$ and hence
\be
\sup_t \sup_s \|A[F,v](E,x,t,s)\| \leq e^{C(E)},
\ee
so that every eigenfunction with such an energy must be bounded.

\end{rem}

%% file: discrete.tex
\section{Discrete case: unbounded eigenfunctions}

\subsection{Schr\"odinger cocycles}

Given a function $V:\Z \to \R$, we define the transfer matrices
$A[V](E,m,n)$ so that $A[V](E,m,m)=\id$,
\be
A[V](E,m,n+1)=\begin{pmatrix}
E-V(n) & -1 \\ 1 & 0 \end{pmatrix} A[V](E,m,n).
\ee
An eigenfunction of the Schr\"odinger operator with potential $V$ and energy
$E$ is a solution of $\begin{pmatrix} \uu_n \\ \uu_{n-1} \end{pmatrix}=
A[V](E,m,n) \cdot \begin{pmatrix} \uu_m \\ \uu_{m-1} \end{pmatrix}$.

\begin{lemma}

If $n>m$ and $|\tr A[V](E,m,n)|<2$ then
\be
\frac {d} {dE} \Theta(A[V](E,m,n))<0.
\ee

\end{lemma}

Assume now that $V$ is periodic of period $N$.  In this case we write
$A[V](E,n)=A[V](E,n,n+N)$ and $A[V](E)=A[V](E,0)$.  Note that $\tr
A[V](E,n)=\tr A[V](E)$ for all $n \in \N$.
Then the spectrum $\Sigma=\Sigma(V)$
of the Schr\"odinger operator with potential $V$ is the set of all $E$ with
$|\tr A[V](E)| \leq 2$.  Let also
$\Omega=\Omega(V)$ be the set of all $E$ with $|\tr A[V](E)|<2$.  We note
that $\Sigma \setminus \Omega=\partial \Omega$ consists of finitely many
points.  For $E \in \Omega(V)$, put
$u[V](E,n)=\u(A[V](E,n))$ and $u[V](E)=u[V](E,0)$.

Let $f:X \to X$ be a minimal
uniquely ergodic map with invariant probability measure
$\mmu$.  Given $v:X \to \R$ continuous, we let $A[f,v](E,x,m,n)=A[V](E,m,n)$
where $V(n)=v(f^n(x))$.  We define the Lyapunov exponent
\be
L(E)=\lim_{n \to \infty} \frac {1} {n} \int \ln \|A[f,v](E,x,0,n)\| d\mmu(x).
\ee

We will use the following criterion for the existence of ac spectrum
(Ishii-Pastur, Kotani, Last-Simon).

\begin{thm}[see \cite {D}]

The ac part of the spectral measure of the discrete Schr\"odinger operator
with potential $V(n)=v(f^n(x))$ is equivalent to the restriction of
Lebesgue measure to $\{L(E)=0\}$.

\end{thm}

\begin{rem}

The fact that the essential support of the ac spectrum is contained in
$\{L(E)=0\}$, for almost every $x$, is the Ishii-Pastur Theorem.  Kotani's
Theorem gives the reverse inclusion, still for almost every $x$.  Those
results apply for general ergodic dynamics.  Last-Simon
proved that for minimal dynamics the essential support of the ac spectrum is
constant everywhere (and not only almost everywhere).

\end{rem}

\subsection{Construction of families of periodic potentials}

\def\V{\mathcal {V}}

Let $\V:\R/N_0 \Z \times \Z/N_1 \Z \to \R$
be a continuous function.  We think of
$\V$ as a one-parameter family (parametrized by $\R/N_0 \Z$)
of periodic potentials $\V_t(\cdot)=\V(t,\cdot)$
(of period $N_1$).  

We define some basic operations on such a $\V$.  First, for $n \in \N$, the
$n$-repetition $\V':\R/N_0 \Z \times \Z/n N_1 \Z \to \R$ of $\V$ is given by
$\V'(t,j)=\V(t,j)$.  We obviously have
\be
A[\V'_t](E,m)=A[\V_t](E,m)^n.
\ee

Secondly, given some $n \in
\N$, we define the $n$-twist $\V':\R/N_0 \Z \times \Z/n N_1 \Z \to \R$ of $\cV$
by $\V'(t,j)=\V(t+N_0 \frac {k} {n},l)$, whenever
$j=k N_1+l$ with $0 \leq j \leq
n-1$ and $0 \leq l \leq N_1-1$.  We notice that
\be
A[\V'_t](E)=A[\V_{t+N_0 \frac {n-1} {n}}](E) \cdots
A[\V_t](E).
\ee

For the third operation, we will make use of some fixed smooth function
$\Psi:[-1,2] \to [0,1]$, with $\Psi=0$ in a neighborhood of $-1$ and $2$,
and $\Psi=1$ in a neighborhood of $[0,1]$.  We also assume that $N_1 \geq
3$.  Then for $\delta>0$ and $n \in \N$, we define the $(\delta,n)$-slide
$\cV':\R/2 n N_0 \Z \times \Z/3 N_1 \Z \to \R$ of $v$ by
\be
\V'(t,j)=\V(t,j), \quad 0 \leq j \leq 2 N_1-1,
\ee
\be
\V'(t,j)=\V(t,j), \quad 2 N_1 \leq j \leq 3 N_1-1 \quad
t \in [0,n N_0-1] \cup [n N_0+2,2 n N_0],
\ee
and
\be
\V'(t,j)=\V(t+\delta \Psi(t-n N_0),j), \quad 2 N_1 \leq j \leq 3 N_1-1
\quad t \in [n N_0-1,n N_0+2].
\ee
Notice that we have
\be
A[\V'_t](E)=A[\V_t](E)^3, \quad t \in [0,n N_0-1] \cup [n
N_0+2,2 n N_0],
\ee
\be
A[\V'_t](E)=A[\V_{t+\delta \Psi(t-n N_0)}](E) \cdot
A[\V_t](E)^2, \quad t \in [n N_0-1,n N_0+2].
\ee

\begin{lemma}

Fix some closed interval $J \subset \R$ and let $u_0:J \times
[-1,2] \to \H$ be a smooth function with
\be \label {u0}
\sup_{t \in [0,1]}
\left |\frac {d} {dt} u_0(E,t) \right |>0
\ee
for every $E \in J$.
There exists $\epsilon_1>0$, $C'>0$ and $\delta_0>0$
with the following property.
Let $\V:\R/N_0 \Z \times \Z/N_1 \Z \to \R$ be a smooth function.
Let $E_0 \in \inter J \cap \bigcap_t \Omega(\V_t)$.
Assume that $[-1,2] \ni t \mapsto u[\V_t](E_0)$ is (strictly)
$\epsilon_1$-close in the $C^1$-topology to $[-1,2] \mapsto u_0(E_0,t)$.
Then there exists $\epsilon_0>0$ such that for every
$0<\epsilon<\epsilon_0$, for every $\kappa>0$,
for every $0<\delta<\delta_0$, for every $N_2$ sufficiently
large, for every $N_3$ sufficiently large, for every $N_4$ sufficiently
large, for every $N_5$ sufficiently large,
if $\V':\R/2 N_4 N_0 \Z \times \Z/3 N_5 N_3 N_2 N_1 \Z$
is the $N_5$-twist of the $(\delta,N_4)$-slide of the $N_2$-repetition of
the $N_3$-twist of $\V$, then there exists a compact set
\be
\Lambda \subset [E_0-\epsilon,E_0+\epsilon] \cap \bigcap_t \Omega(\V'_t)
\ee
such that
\begin{enumerate}
\item $|\Lambda|>2 (1-C' \delta) \epsilon$,
\item For $E \in \Lambda$, $[-1,2] \ni t \mapsto u[\V'_t](E)$ is
(strictly) $\epsilon_1$-close in the $C^1$-topology to
$[-1,2] \ni t \mapsto u_0(E,t)$.
\item For $E \in \Lambda$,
\begin{align}
& \left |\frac {1} {6 N_5 N_4 N_3 N_2 N_1 N_0}
\sum_{j \in \Z/3 N_5 N_3 N_2 N_1 \Z}
\int_0^{2 N_4 N_0} d(u[\V_t](E,j),i) dt \right .\\
\nonumber
&
\left .-\frac {1} {N_1 N_0}
\sum_{j \in \Z/N_1 \Z} \int_0^{N_0} d(u[\V_t](E,j),i) dt \right |<\kappa,
\end{align}
\item For $E \in \Lambda$,
\be
\inf_t \sup_j d(u[\V'_t](E,j),i) \geq \sup_t \sup_j d(u[\V_t](E,j),i)-\kappa,
\ee
\item For any $C' \delta<\gamma<C'^{-1}$, there exists a compact set
$\Lambda' \subset \Lambda$ with $|\Lambda'|>\gamma \epsilon$
such that for $E \in \Lambda'$,
\be
\inf_t \sup_j d(u[\V'_t](E,j),i) \geq
\sup_t \sup_j d(u[\V_t](E,j),i)+C'^{-1} \frac {\delta} {\gamma}-\kappa.
\ee
\end{enumerate}

\end{lemma}

\begin{proof}

Write $\V''''$ for the $N_3$-twist of $\V$, $\V'''$ for the
$N_2$-repetition of $\V''''$, $\V''$ for the $(\delta,N_4)$-slide
of $\V'''$.


In the first step, going from $\V$ to $\V''''$, we obtain, using Lemma \ref
{blaparameter}, a set of
good energies $E \in [E_0-\epsilon,E_0+\epsilon]
\cap \bigcap_t \Omega(\V''''_t)$ of measure at least $2 \epsilon (1-\delta)$,
such that $t \mapsto u[\V''''_t](E)$
is $C^1$ close to $t \mapsto u[\V_t](E)$, and letting
$B(E,t)=B(A[\V''''_t](E))$, $\theta(E,t)=\Theta(A[\V''''_t](E))$,
we have $N_2 (\sup_t \theta(E,t)-\inf_t \theta(E,t))$ arbitrarily small.
Moreover, the random variables $\theta(E,0)$
near $E_0$ are becoming equidistributed in $\R/\Z$ as $N_3$ grows.
We also have that
\be \label {bla}
\inf_t \sup_j d(u[\V''''_t](E,j)) \geq \sup_t \sup_j d(u[\V_t](E,j),i)-\frac
{\kappa} {3}.
\ee

The second step does
not change much, since $E \in \Omega(\V'''_t) $ provided
$N_2 \theta(E,t)$ is not
an integer, and in this case
$u[\V'''_t](E,j)=u[\V''''_t](E,j)$.  On the other hand,
if $\{j \theta(E,0)\}_{0 \leq j \leq N_2-1}$ is $\frac {1}
{100}$ dense $\mod 1$, then for any $w \in \H$, and for every $t$,
\be \label {bla1}
\sup_{0 \leq j \leq N_3 N_2 N_1-1}
d(A[\V'''_t](E,0,j) \cdot w,i) \geq \sup_l d(u[\V'''_t](E,l),i)+\frac {1}
{2} d(w,u[\V'''_t](E)).
\ee
The condition on $\theta(E,0)$ demands the exclusion of
some energies, but of arbitrarily small measure (which we
can take less than $2 \delta \epsilon$),
provided $N_2$ is large.

Going from $\V'''$ to $\V''$, for such good energies $E$ we obtain
\begin{align}
\tr A[\V''_t](&E)=
2 \cos 2 N_2 \pi (2 \theta(E,t)+\theta(E,t+\delta \tilde \Psi(t)))\\
\nonumber
&
-(\lambda(E,t)-\lambda(E,t)^{-1})^2
\sin 4 N_2 \pi \theta(E,t) \sin 2 N_2 \pi
\theta(E,t+\delta \tilde \Psi(t)),
\end{align}
with
\be
\lambda(E,t)=e^{d(u[\V'''_t](E),
u[\V'''_{t+\delta \tilde \Psi(t)}](E))/2},
\ee
and $\tilde \Psi:\R/2 N_4 N_0 \Z \to [0,1]$
is given by $\tilde \Psi(t)=0$ if
$t \in [0,N_4 N_0-1] \cup [N_4 N_0+2,2 N_4 N_0]$ and
$\tilde \Psi(t)=\Psi(t-N_4 N_0)$ if
$t \in [N_4 N_0-1,N_4 N_0+2]$.

Notice that $\lambda(E,t)-1$ vanishes if $t \in [0,N_4 N_0-1] \cup [N_4
N_0+2,2 N_4 N_0]$, is at most of order $\delta$ everywhere, and gets
to be of precisely order $\delta$ for some $t \in [N_0,N_0+1]$ (here we use
(\ref {u0})).

We now exclude $E$ with $\sin 6 N_2 \pi \theta(E,0)$ of order
$\delta$.  The excluded set of energies has measure of order $2 \delta
\epsilon$.  For the remaining energies,
$|\tr A[\V''_t](E)|<2-\delta^2$ for all $t$.

By (\ref {bla1}), for every $t$, using that
$u[\V''_t](E,j)=A[\V''_t](E,0,j)
\cdot u[\V''_t](E)$ and also
$A[\V''_t](E,0,j)=A[\V'''_t](E,0,j)$ for $0 \leq j
\leq 2 N_3 N_2 N_1$, we have
\be
\sup_j d(u[\V''_t](E,j),i) \geq \sup_j d(u[\V'''_t](E,j),i)+\frac {1} {2}
d(u[\V''_t](E),u[\V'''_t](E)),
\ee
and together with (\ref {bla}) we get
\be
\sup_j d(u[\V''_t](E,j),i) \geq \sup_t \sup_j d(u[\V_t](E,j),i)+\frac {1} {2}
d(u[\V''_t](E),u[\V'''_t](E))-\frac {\kappa} {3}.
\ee
In particular, we always have
\be
\sup_j d(u[\V''_t](E,j),i) \geq
\sup_t \sup_j d(u[\V_t](E,j),i)-\frac {\kappa} {3}.
\ee

We compute the distance from $u[\V'''_t](E)$ to $u[\V''_t](E)$.
It is equal to
the distance from $w'(E,t)$ to $i$ where $w'(E,t)$ is the solution $z \in
\H$ of the equation $a z^2+b z+c=0$, where
\begin{align}
a=&\cos 2 N_2 \pi \theta(E,t+\delta \tilde \Psi(t))
\sin 4 N_2 \pi \theta(E,t)\\
\nonumber
&
+\lambda(E,t)^{-2} 
\sin 2 N_2 \pi \theta(E,t+\delta \tilde \Psi(t))
\cos 4 N_2 \pi \theta(E,t),
\end{align}
\be
b=(\lambda(E,t)^2-\lambda(E)^{-2})
\sin 2 N_2 \pi \theta(E,t+\delta \tilde \Psi(t))
\sin 4 N_2 \pi \theta(E,t)),
\ee
\begin{align}
c=&\cos 2 N_2 \pi \theta(E,t+\delta \tilde \Psi(t))
\sin 4 N_2 \pi \theta(E,t)\\
\nonumber
&
+\lambda(E,t)^2 \sin 2 N_2 \pi \theta(E,t+\delta \tilde \Psi(t))
\cos 4 N_2 \pi \theta(E,t)=0.
\end{align}

If the distance from $N_2 \theta(E,0)$ to $\frac {1} {3}+\Z$
is exactly $\gamma$, with
$C_2 \delta<\gamma<C_2^{-1}$, then
\be
C_3^{-1} \frac {\lambda(E,t)^2-\lambda(E,t)^{-2}} {\gamma} \leq
d(w'(E,t),i) \leq C_3 \frac {\lambda(E,t)^2-\lambda(E,t)^{-2}} {\gamma}.
\ee
Using that
$\lambda(E,t)-1$ does become of order $\delta$ for some $t$, we get, for
such $E$,
\be
\sup_t \sup_j d(u[\V''_t](E,j),i) \geq \sup_t \sup_j
d(u[\V_t](E,j),i)+C_4^{-1} \frac {\delta} {\gamma}-\frac {\kappa} {3}.
\ee

On the other hand, if one only excludes the
energies with $\sin 6 N_2 \pi
\theta(E,0)$ of order $\delta$, we still get
that $d(w'(E,t),i)$ is uniformly bounded
as $N_4$ grows, which implies that $\sup_t \sup_j d(u[\V''_t](E,j),i)$
is uniformly bounded as $N_4$ grows.

Proceeding with the last step, we get
\be
\inf_t \sup_j d(u[\V'_t](E,j),i)
\geq \sup_t \sup_j d(u[\V_t](E,j),i)-\frac {2 \kappa} {3},
\ee
while for $2 C_2 \delta<\gamma<C_2^{-1}$ and
a set of $E$ of probability of order $\gamma$ we get
\be
\inf_t \sup_j d(u[\V'_t](E,j),i) \geq
\sup_t \sup_j d(u[\V_t](E,j),i)+C_4^{-1}
\frac {\delta} {\gamma}-\frac {2 \kappa} {3}.
\ee

It remains to check that the average of $d(u[\V'_t](E,j),i)$ is close to the
average of $d(u[\V_t](E,j),i)$.  The first, second, and fourth steps
clearly do not increase the average significantly.  For the
third step, we have $u[\V''_t](E,j)=u[\V'''_t](E,j)$ except when
$t \in [N_4 N_0-1,N_4 N_0+2]$.
Since $d(u[\V''_t](E,j),i)$
remains bounded as $N_4$ grows, we conclude that the
average can not be increased significantly in the third step as well.
\end{proof}

With this result in hands, analogues of Lemmas \ref {bla21} and \ref {bla22}
can be easily obtained.  We state the conclusion:

\begin{lemma} \label {asd12}

Fix some closed interval $J \subset \R$ and let $u_0:J \times
[-1,2] \to \H$ be a smooth function with
\be
\sup_{t \in [0,1]}\left |\frac {d} {dt} u_0(E,t) \right |>0
\ee
for every $E \in J$.
There exists $\epsilon_1>0$ with the following property.
Let $\V^{(0)}:\R/N^{(0)}_0 \Z \times \Z/N^{(0)}_1 \Z \to \R$ be a smooth
function.  Let $\Gamma_0 \subset J \cap \bigcap_t \Omega(\V^{(0)}_t)$
be a compact set of $E$ such that $[-1,2] \ni t
\mapsto u[\V^{(0)}_t](E)$ is (strictly) $\epsilon_1$-close
to $[-1,2] \ni t \mapsto u_0(E,t)$ in the $C^1$-topology.

Let $C>0$ be such that
\be
\sup_{E \in \Gamma_0} \frac {1} {N_0 N_1} \sum_{j=0}^{N_1-1}
\int_0^{N_0} d(u[\V^{(0)}](E,j),i) dt<C.
\ee
Then for every $\xi>0$, $C_0>0$,
for every $\delta>0$ sufficiently small,
there exist $0<P<\xi \delta^{-1}$,
and sequences $N^{(j)}_l$, $1 \leq j \leq P$, $2 \leq l \leq 5$, such that if
we define $\V^{(j)}$, $1 \leq j \leq P$ so that $\V^{(j)}$
is obtained by $N^{(j)}_5$-twist of the $(\delta,N^{(j)}_4)$-slide of the
$N^{(j)}_2$-repetition of the $N^{(j)}_3$-twist of
$\V^{(j-1)}$, then there exists a compact
subset $\Gamma \subset \Gamma_0 \cap \bigcap_t \Omega(\V^{(P)}_t)$
such that $|\Gamma_0 \setminus \Gamma|<\xi$,
and for every $E \in \Gamma$,
letting $N'_0=N_0 2^P \prod_{j=1}^P N^{(j)}_4$
and $N'_1=N_1 3^P \prod_{j=1}^P N^{(j)}_5 N^{(j)}_3 N^{(j)}_2$,
we have
\be
\inf_t \sup_j d(u[\V^{(P)}_t](E,j),i) \geq C_0,
\ee
\be
\frac {1} {N'_0 N'_1} \sum_{j=0}^{N'_1-1}
\int_0^{N'_0} d(u[\V^{(P)}_t](E,j),i)<C.
\ee
Moreover, for $E \in \Gamma$,
$[-1,2] \ni t \mapsto u[\V^{(P)}_t](E)$ is (strictly)
$\epsilon_1$-close
to $[-1,2] \ni t \mapsto u_0(E,t)$ in the $C^1$-topology.
   
\end{lemma}

\begin{rem} \label {asd5}

In the setting of the previous lemma, we have the following extra
information on $\V^{(P)}$.  There exists $n \in \N$ such that for every $E
\in \Gamma$ we have
\be
\inf_{w \in \R^2,\|w\|=1} \sup_{0 \leq l \leq n}
A[\V^{(P)}_t](E,j,j+l)>e^{(C_0-2 C)/4}/2,
\ee
except for a set of $(t,j)$ of measure less than
$C_0^{-1/2}$.  Indeed, if
$(t,j)$ is such that
$d(u[\V^{(P)}_t](E,j),i) \leq C_0^{1/2} C$ and $l \in \N$
is such that $d(u[\V^{(P)}_t](E,j+l),i) \geq C_0$, then
$A[\V^{(P)}_t](E,j,j+l+k N'_1)$ decomposes as a product $B(t,j+l)^{-1}
R_{\tilde \theta+k \theta} B(t,j)$, where
$B(t,m)=\B(A[\V^{(P)}_t](E,m))$, and
$\theta=\Theta(A[\V^{(P)}_t](E,j))$.  Since $2 \theta \notin \Z$, this
implies that for any $w$ we can find $k$ such that
$R_{\tilde \theta+k \theta} B(t,j) \cdot w$ has angle at most $\pi/4$ with
the direction most expanded by $B(t,j+l)^{-1}$, which gives the estimate
$\|B(t,j+l)^{-1} R_{\tilde \theta+k \theta} B(t,j) \cdot w\| \geq
\|B(t,j+l)\| \|B(t,j)\|^{-1}/\sqrt {2}$.

\end{rem}

\subsection{Construction of almost periodic dynamics}

Let $N_0,N_1 \in \N$, and let $a \in \Q$ be an integer multiple of $\frac
{N_0} {N_1}$.  Consider a
smooth family of periodic potentials $\V:\R/N_0 \Z \times \Z/N_1 \Z \to \R$.
It is
natural to consider this periodic family as arising from the non-ergodic
dynamics $(t,j) \mapsto (t,j+1)$ on $\R/N_0 \Z \times \Z/N_1 \Z$, in the
obvious way.  But we can also think of it as arising from the dynamics
$(t,j) \mapsto (t+a,j+1)$, by considering the sampling function
$v:\R/\N_0 \Z \times \Z/N_1 \Z \to \R$ defined by
$\V(t,j)=v(t+j a,j)$.

Such a point of view is advantageous in that it allows to consider
our three
operations on potentials as ``small perturbations''.

Take $\V'$ to be the $n$-repetition of $\V$.  Defining
$v':\R/N_0 \Z \times \Z/n N_1 \Z \to \R$ by
$\V'(t,j)=v'(t+j a,j)$, we obviously still get $v'(t,j)=v(t,j)$.

Take $\V'$ to be the $n$-twist of $\V$.  Set $a'=a+\frac {N_0} {n N_1}$.
Defining $v':\R/N_0 \Z \times \Z/n N_1 \Z \to \R$ by $\V'(t,j)=v'(t+j a',j)$,
we see that
$\sup_t \sup_j |v'(t,j)-v(t,j)|$ becomes small for large $n$.

Take $\V'$ to be the $(\delta,n)$-slide of $\V$.
Defining $v':\R/2 n N_0 \Z \times \Z/3 N_1 \Z \to \R$ by
$\V'(t,j)=v'(t+j a,j)$,
we see that $\sup_t \sup_j |v'(t,j)-v(t,j)| \leq \delta
\sup_t \sup_j |\frac {d} {dt} v(t,j)|$.  Moreover, we also have
$\sup_t \sup_j |\frac {d} {dt} v'(t,j)| \leq (1+K_1 \delta) \sup_t
\sup_j |\frac {d} {dt} v(t,j)|$, where $K_1=\sup_t |\frac {d} {dt}
\Psi(t)|$ is a fixed constant.

Given those observations, we can proceed with the inductive construction.

\noindent {\it Proof of Theorem \ref {discreteunbounded}.}
Choose $0<\lambda_0<1/2$, and let $J=[-2+4 \lambda_0,2-4 \lambda_0]$.
Let $u_0(E,t)$ be the fixed point of $\begin{pmatrix} E-2 \lambda_0 \cos 2
\pi t & -1 \\ 1 & 0 \end{pmatrix}$.  Let $\epsilon_1>0$ be as in Lemma \ref
{asd12}.  Let $C_1>\int_0^1 d(u_0(E,t),i) dt$.

We now produce sequences
$\V^j,v^j:\R/N_{0,j} \Z \times \Z/N_{1,j} \Z \to \R$, $a_j \in \Q$, and
compact sets $\Gamma_j$ as follows.

First set $N_0=N_1=1$, $\V^0(t,j)=2 \lambda \cos 2 \pi t$, $v^0=\V^0$,
$a_0=0$, $\Gamma_0=J$.

We now apply Lemma \ref {asd12}, to obtain
$\Gamma_1 \subset \Gamma_0$ with $|\Gamma_0 \setminus \Gamma_1|$ arbitrarily
small, and some
$\V^1:\R/N_{0,1} \Z \times \Z/N_{1,1} \Z \to \R$ such that for $E \in \Gamma_1$
\be
\inf_t \sup_j d(u[\V^1_t](E,j),i) \geq 2 C_1+4,
\ee
\be
\frac {1} {N_{0,1} N_{1,1}} \sum_{j=0}^{N_{1,1}-1}
\int_0^{N_{0,1}} d(u[\V^1_t](E,j),i)<C_1,
\ee
and moreover, $[-1,2] \ni t \mapsto u[\V^1_t](E,j))$ is
(strictly) $\epsilon_1$-close to $[-1,2] \ni t \mapsto u_0(E,t)$ in the
$C^1$-topology.  Using Remark \ref {asd5}, we see that there exists
$q_1 \in \N$ such that for every $E \in \Gamma_1$,
\be
\inf_{w \in \R^2,\|w\|=1} \sup_{0 \leq l \leq q_1}
\|A[\V^1_t](E,j,j+l) \cdot w\|>\frac {e} {2},
\ee
except for a set of $(t,j)$ of measure
less than $(2 C_1+4)^{-1/2}$.

Moreover, we can alternatively realize $\V^1(t,j)=v^1(t+a_1 j,j)$ for
appropriate $a_1$, so that $\sup_t \sup_j |v^1(t,j)-
v^0(t,j)|$ is arbitrarily small.  Notice that $|a_1-a_0|$ can be also taken
arbitrarily small but non-zero.

We continue by induction, obtaining a decreasing sequence $\Gamma_k$, and
$\V^k,v^k$, $a_k$ such that for $k \geq 2$ and $E \in \Gamma_k$ we
have
\begin{enumerate}
\item $\inf_t \sup_j d(u[\V^k_t](E,j),i) \geq 2 C_1+4 k$,
\item $\frac {1} {N_{0,k} N_{1,k}} \sum_{j=0}^{N_{1,k}-1}
\int_0^{N_{0,k}} d(u[\V^k_t](E,j),i)<C_1$,
\item $[-1,2] \ni t \mapsto u[\V^k_t](E,j)$ is
(strictly) $\epsilon_1$-close to $[-1,2] \ni t \mapsto u_0(E,t)$ in the
$C^1$-topology,
\item $\sup_t \sup_j |v^k(t,j)-v^{k-1}(t,j)|<2^{-k}$,
\item There exists $q_k \in \N$ such that for every $E \in \Gamma_k$,
\be
\inf_{w \in \R^2,\|w\|=1} \sup_{0 \leq l \leq q_k}
\|A[\V^k_t](E,j,j+l) \cdot w\|>\frac {e^k} {2},
\ee
except for a set of $(t,j)$ of measure (strictly) less than
$(2 C_1+4 k)^{-1/2}$.
\item $|a_k-a_{k-1}|$ is non-zero but smaller than
$\frac {1} {2^{k-1} N_{1,k-1}}$.
\end{enumerate}

We now construct the sampling function and the dynamics.

Let $K$ be the projective limit of $\Z/N_{1,j} \Z$ (a Cantor group), and let
$S$ be the projective limit of $\R/N_{0,j} \Z$ (a solenoid).  Then
$v(t,j)=\lim v^k(t,j)$ defines a continuous function on
$S \times K$ (for simplicity, we ommit the projections $S \times K \to
\R/N_{0,k} \Z \times \Z/N_{1,k} \Z$ from the notation).
This is the sampling function.

Notice that $a=\lim a_k$ is irrational, since $a_k$ are rational with
denominators at most $N_{1,k}$ and $0<|a-a_k| \leq \frac {1} {2^{k-1}
N_{1,k}}$.  Thus
$f(t,j)=(t+a,j+1)$ is a uniquely ergodic translation in the
compact Abelian group $S \times K$.  This is the base dynamics.

Let $\Gamma=\bigcap \Gamma_k$.  By construction, $\Gamma$ is a compact set of
positive Lebesgue measure.

Notice that
\be
\sup_x \sup_{E \in \Gamma} \sup_{0 \leq l \leq n}
\|A[f,v](E,x,0,l)-A[\V^k_t](E,j,j+l)\|
\ee
(with $(t,j)$ the projection of $x$)
can be made arbitrarily small, for any $n$
chosen after $v^k$ is constructed, but before $v^{k+1}$ is constructed.
Choosing parameters growing sufficiently fast we get
\be
\lim_{n \to \infty} \sup_x \frac {1} {n} \ln \|A[f,v](E,x,0,n)\|=0,
\ee
i.e., the Lyapunov exponent vanishes for $E \in \Gamma$, so that $\Gamma$ is
contained in the essential support of the ac spectrum for every $x$, and
moreover, for every $k \geq 1$ and $E \in \Gamma_k$,
\be
\inf_{w \in \R^2,\|w\|=1} \sup_{0 \leq l \leq q_k}
\|A[f,v](E,x,0,l) \cdot w\| \geq \frac {e^k} {4},
\ee
except for a set of $x$ of measure less than
$(2 C_1+4 k)^{-1/2}$.
Thus, for every $E \in \Gamma$, for almost every $x$ we have
\be
\inf_{w \in \R^2, \|w\|=1} \limsup_{l \to \infty}
\|A[f,v](E,x,0,l) \cdot w\|=\infty,
\ee
which is the desired eigenfunction growth.
\qed

\section{Discrete case: breaking almost periodicity}

\subsection{Slow deformation}

The following are variations of Lemmas \ref {blainductive} and \ref
{blaparameter}, and we leave the
proofs for the reader.

\begin{lemma}

Let $J \subset \R$ be a closed interval, let $N \in \N$,
and let $A:J \times \R/\Z \to \SL(2,\R)$ be a smooth function such that $|\tr
A^{(N)}(E,t)|<2$ for $(E,t) \in J \times \R/\Z$, where
\be
A^{(N)}(E,t)=A(E,t+\frac {N-1} {N}) \cdots A(E,t).
\ee
Let $B(E,t)=\B(A^{(N)}(E,t))$ and let $\theta(E,t)$ be a smooth function
satisfying
\be
B(E,t+\frac {1} {N}) A(E,t) B(E,t)^{-1}=R_{\theta(E,t)}.
\ee
Then for every $m,k \in \N$, there exists $n(m) \in \N$ and
$C_{k,m}>0$
such that for every $n \geq n(m)$,
there exist smooth functions
$B_{(m,n)}:J \times \R/\Z \to \SL(2,\R)$, $\theta_{(m,n)}:J \times \R/\Z \to \R$
such that
\begin{enumerate}
\item $\|A_{(m,n)}-R_{\theta_{(m,n)}}\|_{C^k} \leq \frac {C_{k,m}} {n^m}$,
where
\be
A_{(m,n)}(E,t)=B_{(m,n)}(E,t+\frac {n+1} {n N})
A(E,t) B_{(m,n)}(E,t)^{-1},
\ee
\item $\|B_{(m,n)}-B\|_{C^k} \leq \frac {C_{k,m}} {n}$,
\item $\|\theta_{(m,n)}-\theta\|_{C^k} \leq \frac {C_{k,m}} {n}$.
\end{enumerate}

\end{lemma}

\begin{lemma} \label {asd3}

Under the hypothesis of the previous lemma, assume moreover that
$\tilde \theta(E)=\int_{\R/\Z} \theta(E,t) dt$ satisfies
$\frac {d} {dE} \tilde
\theta(E) \neq 0$ for every $E \in J$.
For $n \in \N$, let
$A^{(N * n)}:J \times \R/\Z \to \SL(2,\R)$ be given by
\be
A^{(N * n)}(E,t)=A(E,t+(n N-1) \frac {n+1} {n N})
A(E,t+(n N-2) \frac {n+1} {n N}) \cdots
A(E,t).
\ee
Then there exist functions $\tilde \theta^{(n)}:J \to \R/\Z$ such that for
every measurable subset $Z \subset \R/\Z$,
\be
\lim_{n \to \infty} |\{E \in J,\, \tilde \theta^{(n)}(E) \in Z\}|=|Z| |J|,
\ee
\be
\lim_{n \to \infty} |\{E \in J,\, \tilde \theta^{(n)}(E)+\tilde
\theta^{(2n)}(E) \in Z\}|=|Z| |J|,
\ee
with the following property.  For every $\delta>0$,
\be   
\lim_{n \to \infty} \|\tr A^{(N*n)}(E,t)-2 \cos 2 \pi \tilde
\theta_{(m,n)}(E)\|_{C^0(J \times \R/\Z,\R)}=0,
\ee
\be
\lim_{n \to \infty} \sup_{|\sin 2 \pi \tilde \theta^{(n)}(E)|>\delta}
\|\Theta(A^{(N*n)}(E,\cdot))-
\tilde \theta^{(n)}(E)\|_{C^1(\R/\Z,\R)}=0,
\ee
\be
\lim_{n \to \infty} \sup_{|\sin 2 \pi \tilde \theta^{(n)}(E)|>\delta}
\|\u(A^{(N * n)}(E,\cdot))-\u(A^{(N)}(E,\cdot))\|_{C^1(\R/\Z,\C)}=0.
\ee
   
\end{lemma}


\subsection{The construction}

In this section, we will interpret a continuous function
$\V:\R/N \Z \to \R$ as a one-parameter family of $N$-periodic potentials
$\V_t(j)=\V(t+j)$.

For $n \in \N$, we define the $n$-crumbling $\V':\R/3 n N \Z \to \R$
of $v$ by
\begin{enumerate}
\item $\V'(t)=\V(\frac {n+1} {n} t)$, $t \in [0,n N]$,
\item $\V'(t)=\V(\frac {2 n+1} {2 n} (t-n N))$, $t \in [n N,3 n N]$.
\end{enumerate}

\begin{lemma} \label {asd4}

Let $\V:\R/N \Z$ be a smooth function which is constant near $0$.
Then for every $\delta>0$, for every $n$
sufficiently large, letting $\V'$ be the $n$-crumbling of $\V$, we have
$|\bigcap_t \Omega(\V_t) \setminus \bigcap_t \Omega(\V'_t)|<\delta$.


\end{lemma}

\begin{proof}

Fix a compact
interval $J \subset \bigcap_t \Omega(\V_t)$.
Apply Lemma \ref {asd3} to $A(E,t)=\begin{pmatrix} E-v(N t) & -1 \\ 1 & 0
\end{pmatrix}$.  It yields a sequence $\tilde \theta^{(n)}(E)$.

If $\V'$ is the $n$-crumbling of $v$, then for $t \in [0,1]$, we get
$A[\V'_t](E,0,n N)=A^{(N*n)}(E,\frac {n+1} {n N} t)$ and
$A[\V_t](E,n N,3nN)=A^{(N*2 n)}(E,\frac {2 n+1} {2 n N} t)$. 
Thus for $t \in [0,1]$ we have
\be
A[\V_t](E)=A^{(N*2 n)}(E,\frac {2 n+1} {2 n N} t)
A^{(N*n)}(E,\frac {n+1} {n N} t)
\ee

As long as $|\sin 2 \pi \tilde \theta^{(n)}|$ and $|\sin 2 \pi \tilde
\theta^{(2n)}|$ are not too small, we can write, for $t \in [0,1]$,
\be
A[\V'_t](E)=B^{(2 n)}(E,t)^{-1} R_{\theta^{(2 n)}(E,t)}
B^{(2 n)}(E,t)
B^{(n)}(E,t)^{-1} R_{\theta^{(n)}(E,t)} B^{(n)}(E,t),
\ee
where $B^{(m)}=\B(A^{(N*m)}(E,t))$ and
$\theta^{(m)}(E,t)=\Theta(A^{(N*m)}(E,t))$.
Notice that $B^{(n)}$ and $B^{(2 n)}$ are
both $C^1$-close to $\B(A[\V_t](E))$ as functions of $t \in [0,1]$.
Moreover, $\theta^{(n)}(E,t)$ is close to $\tilde \theta^{(n)}(E)$ and
$\theta^{(2n)}(E,t)$ is close to $\tilde \theta^{(2n)}(E)$.

It follows that for $t \in [0,1]$,
$\tr A[\V'_t](E)$ is close to $2 \cos 2 \pi
(\tilde \theta^{(2 n)}(E)+\tilde \theta^{(n)}(E))$.  
Thus, as long as $|\sin 2 \pi (\tilde \theta^{(n)}+\tilde \theta^{(2n)})|$
is not small, we have $|\tr A[\V'_t](E)|<2$ for every $t \in [0,1]$. 
Since $\tr A[\V'_t](E)$ is $1$-periodic, this implies that $|\tr
A[\V'_t(E)|<2$ for all $t$.
\end{proof}

\begin{rem} \label {blaremark}

One also easily gets from this construction,
\be
\sup_t d(u[\V'_t](E),i)<\sup_t d(u[\V_t](E),i)+\delta
\ee
except for a set of $E \in \bigcap_t \Omega(\V_t) \cap \bigcap_t
\Omega(\V'_t)$ of arbitrarily small measure.

\end{rem}

\noindent {\it Proof of Theorem \ref {continuumwm}.}
Starting with a smooth non-constant function $\V^{(0)}:\R/\Z \to \R$,
apply Lemma \ref {asd4} successively to obtain a sequence
$\V^{(k)}:\R/N^{(k)} \Z \to \R$ such that $\V^{(k)}$ is the
$n_k$-crumbling of $\V^{(k-1)}$, and compact sets $\Gamma^{(k)} \subset
\bigcap_t \Omega(\V^{(k)}_t)$
with $\Gamma^{(k)} \subset \Gamma^{(k-1)}$ and $\lim_{k \to \infty}
|\Gamma^{(k)}|>0$.
By taking parameters $n_k$ growing sufficiently fast, we ensure that for $E
\in \Gamma^{(k+1)}$ we have
\be \label {asd7}
\sup_{1 \leq j \leq N^{(k)}} |\sup_t \ln \|A[\V^{(k+1)}_t(E,0,j)\|-
\sup_t \ln \|A[\V^{(k)}_t](E,0,j)\|| \leq \frac {1} {2^k},
\ee
\be \label {asd8}
\sup_t \frac {1} {N^{(k+1)}}
\ln \|A[\V^{(k+1)}_t](E,0,N^{(k+1)})\| \leq  \frac {1} {2^k}.
\ee

We now turn to the dynamical realization.  Let $\tilde N^{(k)}$ be defined
by $\tilde N^{(0)}=1$, $\tilde N^{(k)}=(3 n_k+2) \tilde N^{(k-1)}$.
We first construct $N^{(k)}$-periodic time changes
$F^{(k)}$ of the solenoidal flow on $S^{(k)}=\R/\tilde N^{(k)} \Z$ such
that $\V^{(0)}(p_{S^{(k)},S^{(0)}}(F^{(k)}_t(0)))=\V^{(k)}(t)$.  We first take
$F^{(0)}$ to be just
the solenoidal flow on $S^{(0)}$.  Now define inductively
\be
w_{F^{(k+1)}}(t)=w_{F^{(k)}} 
(t) e^{\rho^{(k+1)}(t)}
\ee
for a suitable function $\rho^{(k+1)}$.  Here it is enough to take
$\rho^{(k+1)}=\ln \frac {n_{k+1}+1} {n_k}$ on $[0,(n_{k+1}+1)
\tilde N^{(k)}]$, $\rho^{(k+1)}=
\ln \frac {2 n_{k+1}+1} {2 n_{k+1}}$ on $[(n_{k+1}+1) \tilde
N^{(k)}+\epsilon,(3 n_{k+1}+2) \tilde N^{(k)}-\epsilon]$,
for suitably small $\epsilon$, and such that
\be
\int_{(n_{k+1}+1) \tilde N^{(k)}}^{(n_{k+1}+1) \tilde N^{(k)}+\epsilon}
\frac {1} {w_{F^{(k)}}(t)
e^{\rho^{(k+1)}(t)}} dt=\frac {2 n_{k+1}} {2 n_{k+1}+1}
\int_0^\epsilon \frac {1} {w_{F^{(k)}}(t)} dt,
\ee
\be
\int_{(3 n_{k+1}+2) \tilde N^{(k)}-\epsilon}^{(3 n_{k+1}+2)
\tilde N^{(k)}} \frac {1} {w_{F^{(k)}}(t)
e^{\rho^{(k+1)}(t)}} dt=\frac {2 n_{k+1}} {2 n_{k+1}+1} \int_{\tilde
N_k-\epsilon}^{\tilde N^{(k)}} \frac {1} {w_{F^{(k)}}(t)} dt.
\ee
Notice that by taking parameters growing sufficiently fast, we can take
$F^{(k+1)}$ close to the lift of $F^{(k)}$.

Let $S$ be the projective limit
of $\R/\tilde N^{(k)} \Z$, and let $v:S \to \R$ be given by
$v(x)=\V^{(0)}(p_{S,S^{(0)}}(x))$.  Let $F_t:S \to S$ be the projective
limit of the $F^{(k)}_t$.
The base dynamics will be the time-one map $F_1$ and the sampling function
will be $v$.

By (\ref {asd7}), for every $k$, if $E \in \Gamma=\bigcap \Gamma^{(k)}$,
\be \label {blaclose3}
\sup_{1 \leq j \leq N^{(k)}} |\sup_x \ln \|A[F_1,v](E,x,0,j)\|-
\sup_t \ln \|A[\V^{(k)}_t](E,0,j))\|| \leq \frac {1} {2^{k-1}},
\ee
and together with (\ref {asd8}) we get, for $E \in \Gamma$
\be
\sup_x \frac {1} {N^{(k+1)}}
\ln \|A[F_1,v](E,x,0,N^{(k+1)})\| \leq
\frac {1} {2^k}+\frac {1} {N^{(k+1)} 2^k}
\leq \frac {1} {2^{k-1}},
\ee
so that the Lyapunov exponent (with respect to any $F_1$-invariant measure)
must vanish over $\Gamma$.

To conclude, let us show that the flow $F$ is weak mixing: This implies
that the discrete dynamics $F_1$ is weak mixing as well, and since $F$ is
minimal and uniquely ergodic, it also implies that $F_1$ is minimal and
uniquely ergodic, so that $\Gamma$ is contained in the essential
support of the absolutely continuous spectrum for every $x$.
In order to do this, we notice that for $0 \leq j \leq n_{k+1}-1$
\be \label {asd1}
p_{S^{(k+1)},S^{(k)}}(F^{(k+1)}_{j N^{(k)}}(t))=
F^{(k)}_{j/n_{k+1}} (p_{S^{(k+1)},S^{(k)}}(t)),
\ee
as long as $t \in
[0,(n_{k+1}+1)
(1-\frac {j} {n_{k+1}}) \tilde N^{(k)}]$.
On the other hand,
for $0 \leq j \leq 2 n_{k+1}-1$
\be \label {asd2}
p_{S^{(k+1)},S^{(k)}}(F^{(k+1)}_{j N^{(k)}}
(t))=F^{(k)}_{j/2 n_{k+1}}(p_{S^{(k+1)},S^{(k)}}(x)),
\ee
as long as $t \in [(n_{k+1}+1) \tilde N^{(k)}+1,
\tilde N^{(k+1)}-\frac {2 n_{k+1}+1} {2 n_{k+1}}j\tilde N^{(k)}-1]$.

The conclusion proceeds along the same line as in Lemma \ref {blamixed2}.
Take a measurable eigenfunction $\psi$ taking values on the unit circle,
associated to an eigenvalue
$\theta \neq 0$, so that $\psi \circ F_t=e^{2 \pi i \theta t} \psi$.
Taking
conditional expectations, we obtain
$\psi^{(j)}$ on $S^{(j)}$, taking values on the closed unit disk, with
$\lim \psi^{(j)}(p_{S,S^{(j)}}(x))=\psi(x)$
for almost every $x$.  We then conclude from (\ref {asd1}) and (\ref {asd2})
that $\frac {\theta j} {2 n_{k+1}}$ is close to an integer for $1 \leq j
\leq [n_{k+1}/2]$.  This contradicts $\theta \neq 0$.
\qed

\begin{rem} \label {blaremark2}

Using Remark \ref {blaremark}, we can ensure in the construction that
\be
C=\sup_k \sup_{E \in \Gamma_k} \sup_t d(u[\V^{(k)}_t](E),i)<\infty.
\ee
This implies that
\be
\sup_k \sup_{E \in \Gamma_k} \sup_t \sup_j \|A[\V^{(k)}_t](E,0,j)\| \leq e^C,
\ee
and by (\ref {blaclose3}),
\be
\sup_{E \in \Gamma} \sup_x \sup_j \|A[F_1,v](E,x,0,j)\| \leq e^C,
\ee
so all eigenfunctions with energies in $\Gamma$ are bounded.

\end{rem}